\documentclass[12pt,reqno]{amsart}
\title{On the Parameter Spaces of Harmonic Trinomial Equations}

\usepackage[letterpaper,top=3cm,bottom=3cm,left=3cm,right=3cm,marginparwidth=1.75cm]{geometry}
\usepackage[hidelinks]{hyperref}

\usepackage{amsthm}
\usepackage{amssymb}
\usepackage{amsmath}
\usepackage{graphicx}
\usepackage{natbib}
\usepackage{subcaption}
\newtheorem{theorem}{Theorem}[section]
\newtheorem{lemma}[theorem]{Lemma}
\newtheorem{proposition}[theorem]{Proposition}
\newtheorem{definition}[theorem]{Definition}
\newtheorem{example}[theorem]{\textit{Example}}
\newtheorem{corollary}[theorem]{corollary}

\newtheorem{remark}[theorem]{Remark}

\title{On the Parameter Spaces of Harmonic Trinomial Equations}

\author{Waldemar Barrera}
\address{Facultad de Matemáticas, UADY.}
\email{bvargas@correo.uady.mx}

\author{Lucía Campa}
\address{Facultad de Matemáticas, UADY.}
\email{lucia.campa@alumnos.uady.mx}

\author{Juan Pablo Navarrete}
\address{Facultad de Matemáticas, UADY.}
\email{jpnavarrete@correo.uady.mx}

\date{\today}

\begin{document}

\begin{abstract}
We analyze the parameter space of harmonic trinomial equations of the form
$z^{n+m}+b\overline{z}^m+c$, where $n,m\in\mathbb{Z}^+$ are coprime and $b,c\in\mathbb{C}$.
Using versions of the Bohl and Egerváry theorems for harmonic trinomials, we describe the geometric curves in the parameter space that arise when considering a simple root or a multiple root, or when two distinct roots have the same modulus.
In particular, we study the geometric properties of these curves, called trochoids.
\end{abstract}

\maketitle

\section*{Introduction}\label{sec1}

The study of trinomial equations is a classical topic in mathematics that has given rise to numerous and relevant investigations over the years \cite{bohl1908theorie}, \cite{egervary1930trinom}, \cite{euler2}, \cite{euler}. These equations continue to be an active area of research. By example, in 2019, Belkić \cite{Belkic} compiled more than 250 years of theory on trinomial equations and their applications in the modeling of complex phenomena.

Recently, Thorsten and De Wolff studied in \cite{theobald2016norms} the parameter space of trinomial equations, as well as several geometric and topological properties that had not been previously revealed. In that study, the space of trinomial equations of the form
\begin{equation}
\label{trinomial}
    z^{n+m}+bz^m+c=0
\end{equation}
is denoted by $T_A$, where $b$ and $c$ are complex numbers that are not simultaneously zero. Based on this framework, the authors posed and answered the following questions:

\begin{enumerate}
    \item [(I)] What is, for given $c$, the geometric structure of the set of all $b$ such that \eqref{trinomial} has a root with norm $v>0$?
    \item [(II)] What is, for given $c$, the geometric structure of the set of all $b$ such that \eqref{trinomial} has two roots of norm $v>0$?
     \item[(III)]
    Denote by $U_j^A$ the subset of trinomials in $T_A$ whose $j$-th and $(j+1)$-th smallest root (ordered by their norm) have distinct norm, $1\leq j\leq n+m-1$. Formally, they also consider $U_0^A$ and $U_{n+m}^A$, by declaring $f\in U_0^A$ and $f\in U_{n+m}^A$ for every trinomial $f\in T_A$.
   So which geometric and topological properties do the sets $U_j^A\subseteq T_A$ and their complements have? 
\end{enumerate}

Following this line of research, we may formulate the previous questions (I, II, and III) in the context of harmonic trinomials of the form
$az^{n+m}+b\overline{z}\,^m+c$, with $n,m\in\mathbb{Z}^+$, $a,b,c\in\mathbb{C}$, where an anti-analytic term is added, introducing nontrivial geometric and topological behavior. This approach aligns with broader developments in harmonic function theory, where several applications have been found. For example, in \cite{Khavinson2008}, \cite{khavinson}, Khavinson and Neumann showed a subclass of harmonic polynomials related to multiple images produced by gravitational lenses, which can generate patterns similar to the geometric structures described by harmonic trinomials. Similarly, in \cite{AnJin}, An studies the relationship between weak perturbations affecting the potential of a gravitational lens and the shape of the resulting caustics. For these reasons, it is particularly natural and relevant to try to extend several results previously established for analytic trinomials to a more general setting by introducing an anti-analytic term.

Motivated by the results of Thorsten and De Wolff in Section~4 of \cite{theobald2016norms}, which are based on Amoeba theory and the theorems of Bohl~\cite{bohl1908theorie} and Egerváry~\cite{egervary1930trinom} for analytic trinomial equations, and also motivated by the results of Brilleslyper \emph{et al.}~\cite{brilleslyper}, which rely primarily on Rouché’s theorem and the Argument Principle for Harmonic Functions, we ask what the analogous results are for harmonic trinomial equations. 

Let $H$ be the space of all harmonic trinomials such that $n$ and $m$ are relatively prime. If we fix the parameter $c$ for all harmonic trinomials of the form
\begin{equation}
\label{ecs_trin_arm}  h(z)=z^{n+m}+b\overline{z}\,^m+c,
\end{equation}

\begin{itemize}
    \item[(A)] What is the geometric structure of the set of all $b$ such that \eqref{ecs_trin_arm} has a root with norm $v>0$? (Theorem~\ref{teo:param_trocoide}).
    
    \item[(B)] What is the corresponding locus of the set of all $b$ such that \eqref{ecs_trin_arm} have two roots with the same norm $v>0$? (Theorem~\ref{RaizDoble}).
    \item[(C)] If we fix the parameter $b$, what is the curve consisting of the set of all $c$ such that \eqref{ecs_trin_arm} has a root of norm $v>0$? (Theorem~\ref{teo:param_trocoide2}).

\item[(D)] Let $U_j$ denote the subset of all harmonic trinomials in $H$ whose $j$-th and $(j+1)$-th smallest roots (ordered by their norm) have distinct norms, for $1\leq j\leq l$, where $l$ denotes the index of the last root of \eqref{ecs_trin_arm}  with the largest norm in this ordering and satisfies $l\leq n+3m$. What geometric and topological properties do the sets $U_j$ and their complements in $H$ have? (Theorems \ref{teo:union_rayos}, \ref{teo_complementos}, and Corollary \ref{radio}).
\end{itemize}

We notice that question (D) is different to the previous question (III) due to the Fundamental Theorem of Algebra. For trinomial equations of the form \eqref{trinomial}, we introduce the sets $U_j^A$ as in (III), and these equations have exactly $n+m$ roots; therefore, there are at most $n+m+1$ such sets $U_j^A$. For harmonic trinomials, however, this property does not hold. According to \cite{barrera2022number}, any harmonic trinomial of the form \eqref{ecs_trin_arm} has at most $n+3m$ roots. Consequently, there are exactly $n+3m+1$ sets $U_j$.

This work is divided into two sections. In Section~\ref{Preliminares}, we introduce useful tools for the development of this article. Additionally, we include a summary of Section~4 of \cite{theobald2016norms} for trinomial equations, with a subsection on trochoid curves, which are useful for describing the geometry associated with harmonic trinomial equations. Section~\ref{Results} addresses the questions (A, B, C and D) posed by analyzing the local structure of the parameter space obtained by fixing one of the coefficients of the harmonic trinomials.

Finally, we emphasize that this work does not make use of Amoeba theory as in \cite{theobald2016norms}. Instead, the mathematical tools employed to develop the results of this article rely on the Bohl and Egerváry theorems in their versions for harmonic trinomials, which were developed recently in \cite{barrera2022number} and \cite{barrera2024egervary}.

\section{Preliminaries}
\label{Preliminares}
\subsection{Bohl Theorems for Harmonic Trinomials}
For the reader's convenience we include the definition for harmonic trinomial equations and some of the main results of Bohl's Theorem in its versions for harmonic trinomials presented in \cite{barrera2022number}, \cite{BARRERA2024128213}, which are used in some of the proofs in Section \ref{Results}.

\begin{definition}
\label{fnc}
A \emph{harmonic trinomial equation} is an expression of the form
\[
a z^{n+m} + b\,\overline{z}\,^m + c = 0,
\]
where $a,b,c \in \mathbb{C}^*$ and $n,m$ are positive integers.
\end{definition}

\begin{theorem}[Bohl for harmonic trinomials, \cite{barrera2022number}]
\label{Bohlarm1}
Let $h(z)=az^{n+m}+b\overline{z}\,^m+c$ for all $z\in\mathbb{C}$ be a harmonic trinomial with $a, b, c\in\mathbb{C}$. Let $v$ be a positive real number and $k$ the number of roots whose modulus is less than $v$. Then, the following statements hold:
\begin{itemize}
    \item If $|c|>|a|v^{n+m}+|b|v^m$, then $k=0$.
    \item If $|a|v^{n+m}>|b|v^m+|c|$, then $k=n+3m$.
    \item If $|b|v^m>|a|v^{n+m}+|c|$, then $k=m$.
\end{itemize}
\end{theorem}

\begin{example}
    Let $h(z)=z^{4}-5\overline{z}\,^3+2$. For $v =\frac{1}{2},2,6$, we have $k=0,3,10$ roots inside the circle $C_1,C_2,C_3$ respectively, as shown in Figures~\ref{fig:4}, \ref{fig:5} and \ref{fig:6}.

\begin{figure}[h!] 
    \centering 
    \begin{subfigure}[b]{0.3\textwidth} 
        \centering
    \includegraphics[width=\textwidth]{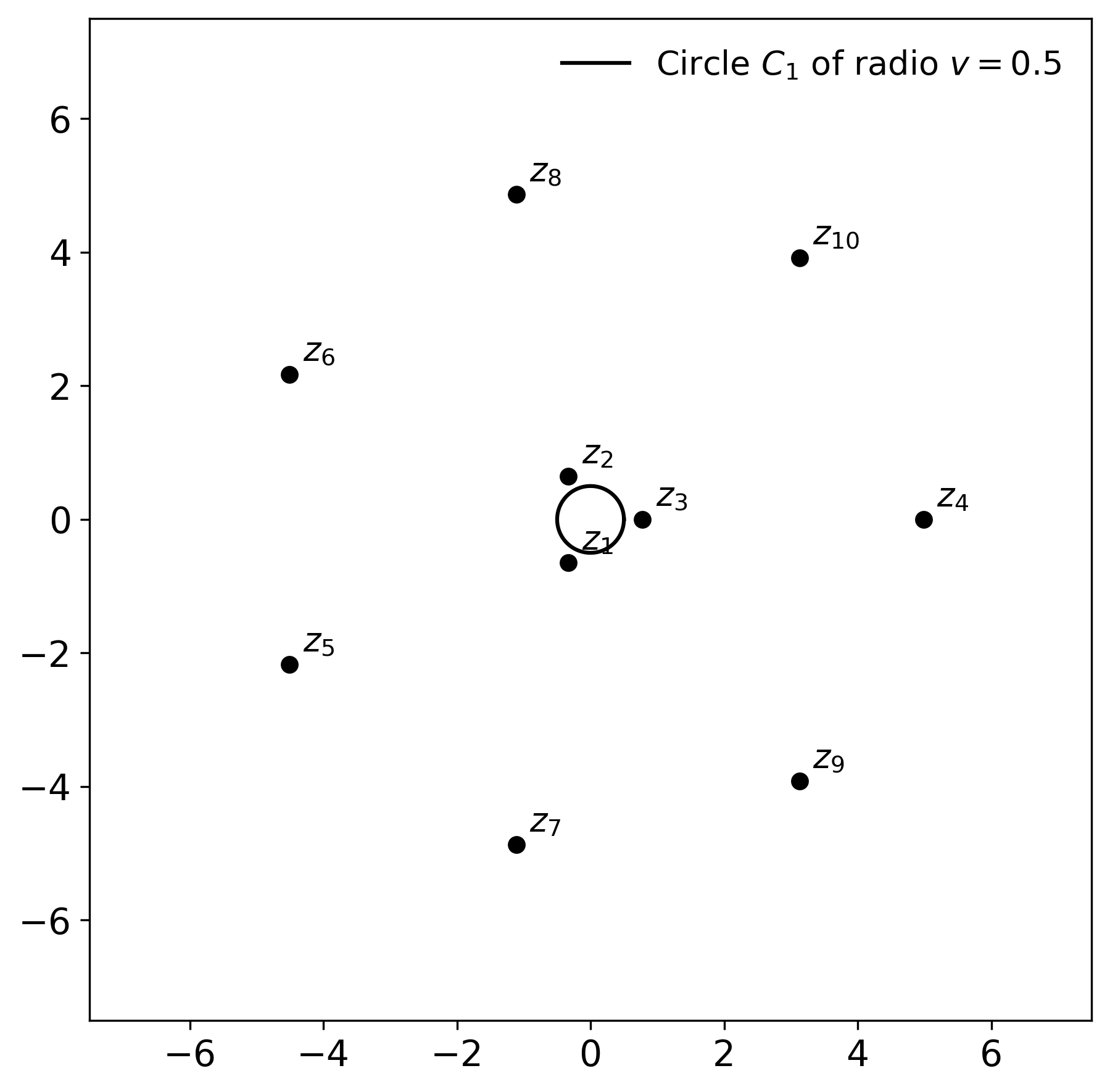} 
        \caption{\small{Inside the circle \( C_1 \), there are \( k = 0 \) roots.}}
        \label{fig:4}
    \end{subfigure}
    \hfill 
    \begin{subfigure}[b]{0.3\textwidth}
        \centering
    \includegraphics[width=\textwidth]{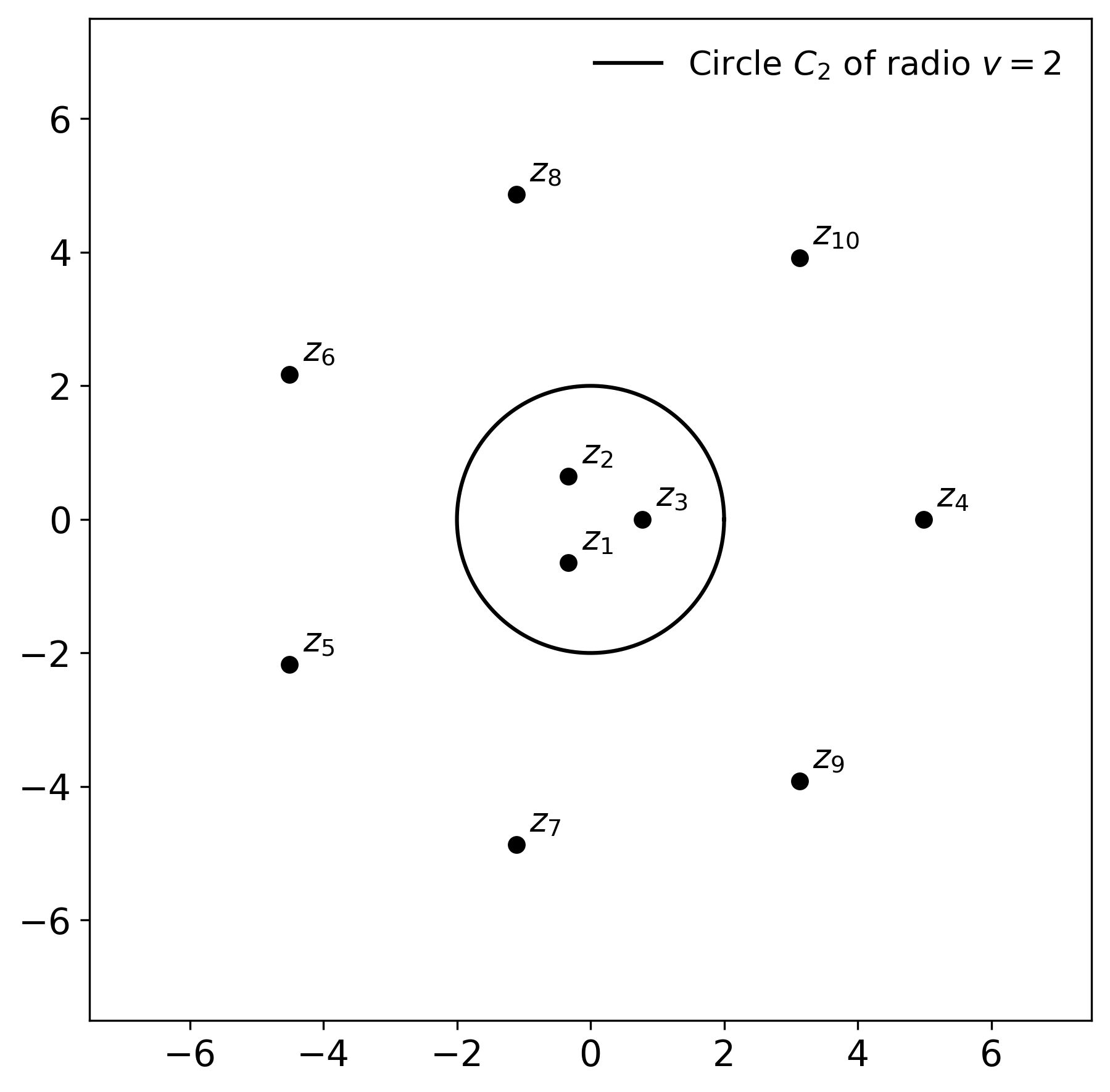}
        \caption{\small{Inside the circle \( C_2 \), there are \( k = 3 \) roots.}}
        \label{fig:5}
    \end{subfigure}
    \hfill
    \begin{subfigure}[b]{0.3\textwidth}
        \centering
    \includegraphics[width=\textwidth]{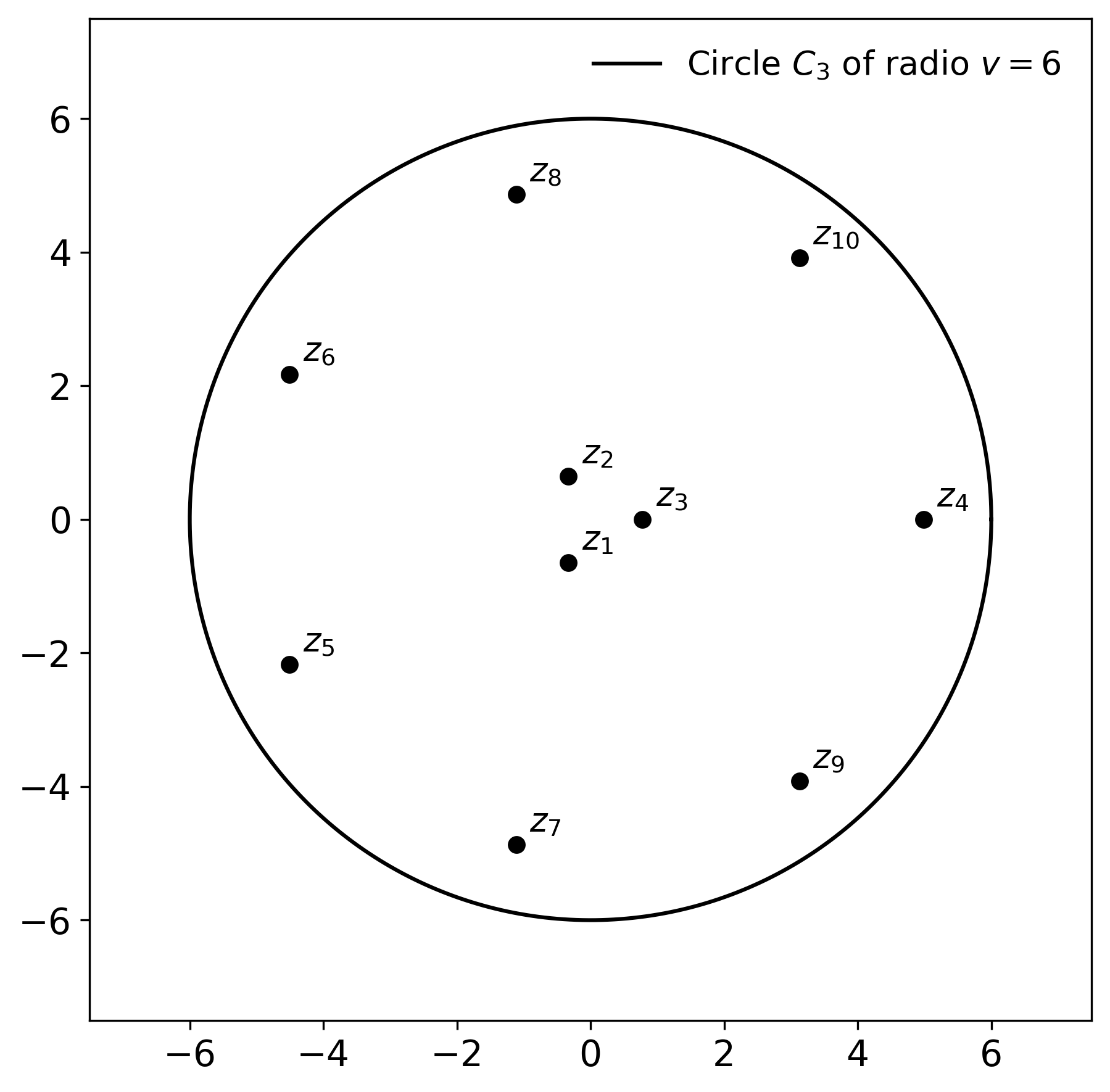}
        \caption{\small{Inside the circle \( C_3 \), there are \( k = 10 \) roots.}}
        \label{fig:6}
    \end{subfigure}
    \caption{\small{Roots $z_1, z_2,\ldots,z_{10}$ of the harmonic trinomial $h(z)=z^4-5\overline{z}\,^3+2$. And the circles \( C_1 \), \( C_2 \) and \( C_3 \) with radio \( v = \tfrac{1}{2}, 2, \text{ and } 6 \) respectively.
}} 
\end{figure}
\end{example}

Theorem \ref{Bohlarm1} analyzes the cases when, for any $v\in\mathbb{R}^+$, the lengths $|a|v^{n+m}$, $|b|v^m$, and $|c|$ associated with the harmonic trinomial $h(z)=az^{n+m}+b\overline{z}\,^m+c$ do not form the sides of a triangle. However, there are cases when these lengths do form the sides of a triangle.
The following elements, defined in \cite{barrera2022number}, will be used in Theorem \ref{BohlArmonico}. Let
\begin{equation}
\label{polinomios}
A(v)=|a|v^{n+m}-|b|v^m-|c|,\,\,\,\,B(v)=-|a|v^{n+m}+|b|v^m-|c|\,\,\,\text{ and }\,\,C(v)=-|a|v^{n+m}-|b|v^m+|c|.
\end{equation}
Let $\mathfrak{a}$ and $\mathfrak{c}$ be the unique positive roots of $A$ and $C$, respectively. Moreover, $\mathfrak{a}$ and $\mathfrak{c}$, where $\mathfrak{c}<v<\mathfrak{a}$, are the endpoints that delimit the region $\mathbb{T}$ where $|a|v^{n+m}$, $|b|v^m$, and $|c|$ form the sides of a triangle. Then, three possible cases arise for $\mathbb{T}$:
\begin{itemize}
\item $\mathbb{T}=(\mathfrak{c},\mathfrak{a})$.
\item $\mathbb{T}=(\mathfrak{c},\mathfrak{b})\cup(\mathfrak{b},\mathfrak{a})$, where $\mathfrak{b}$ is a double root of $h$.
\item $\mathbb{T}=(\mathfrak{c},\mathfrak{b_1})\cup(\mathfrak{b_2},\mathfrak{a})$, where $\mathfrak{b_1}\neq\mathfrak{b_2}$ are roots of $h$.
\end{itemize}
In each of these cases, on the boundary of these intervals, a degenerate triangle is formed. For more details, see \cite{barrera2022number}, \cite{bohl1908theorie}.

Given $v\in\mathbb{R}^+$ and a harmonic trinomial function $h:\mathbb{C}\rightarrow\mathbb{C}$, we define the set $Z_h(v):=\{\zeta\in\mathbb{C}\,|\,h(\zeta)=0 \text{ \,\,and\,\, }|\zeta|<v\}$. We denote the cardinality of a set $X$ by $\text{Card}(X)$, and instead of $\text{Card}(X\cap\mathbb{Z})$, we will write $\text{Card}_{\mathbb{Z}}(X)$ for brevity.

\begin{theorem}[Bohl for harmonic trinomials, \cite{barrera2022number}, \cite{BARRERA2024128213}]
\label{BohlArmonico}
Let $h(z)=az^{n+m}+b\overline{z}\,^m+c$ for all $z\in\mathbb{C}$ be a harmonic trinomial with $a,b,c\in\mathbb{C}$. Let $v>0$, and assume that $|a|v^{n+m}$, $|b|v^m$, and $|c|$ are the sides of some (non-degenerate) triangle. Let $w_1$ and $w_2$ be the angles opposite to the sides of lengths $|a|v^{n+m}$ and $|b|v^m$, respectively. For $\mathfrak{c}<v<\mathfrak{a}$, with $v\neq \mathfrak{b}_j$ where $j\in\{1,2\}$ (including the case $\mathfrak{b}_1=\mathfrak{b}_2$), it follows that
\begin{equation*}
    \text{Card}(Z_h(v))=\text{Card}_{\mathbb{Z}}(P_*-w_*(0,v))+\text{Card}_{\mathbb{Z}}(P_*+w_*(0,v))+\textbf{1}_{\{P_*\in\mathbb{Z}\}},
\end{equation*}
where the pivots $P_*$ and $w_*(v)$ are given by
\begin{equation*}
   P_*:=\frac{(n+m)(\beta-\gamma-\pi)+m(\alpha-\gamma-\pi)}{2\pi}\quad\text{ and }\quad w_*(v):=\frac{(n+m)w_1-mw_2}{2\pi},
\end{equation*}
with
\begin{equation*}
    \begin{array}{ccc}
        P_*-w_*(0,v) & := & \{P_*-w_*(u):u\in(0,v)\}, \\
        P_*+w_*(0,v) & := & \{P_*+w_*(u):u\in(0,v)\},
    \end{array}
\end{equation*}
where $\alpha,\, \beta$, and $\gamma$ are the arguments of $a,\,b$, and $c$, respectively, and $\textbf{1}_{\{P_*\in\mathbb{Z}\}}:=1$ if $P_*$ is an integer, and $\textbf{1}_{\{P_*\in\mathbb{Z}\}}:=0$ otherwise.\\

When $|a|v^{n+m}$, $|b|v^m$, and $|c|$ form the sides of a degenerate triangle, we define 
\begin{equation*}
    w_*(v)=
    \begin{cases}
      0 & \text{if }\ |c|=|a|v^{n+m}+|b|v^m, \\
      \frac{n+m}{2} & \text{if }\ |a|v^{n+m}=|b|v^m+|c|,\\
      -\frac{m}{2} & \text{if }\ |b|v^m=|a|v^{n+m}+|c|,
    \end{cases}
\end{equation*}
and since $w_1$ and $w_2$ are continuous functions of $v$, we can continuously extend $w_*(v)$ as follows:
\begin{equation*}
    w_*(v)=
    \begin{cases}
      0 & \text{if }\ |c|\geq|a|v^{n+m}+|b|v^m, \\
      \frac{n+m}{2} & \text{if }\ |a|v^{n+m}\geq|b|v^m+|c|,\\
      -\frac{m}{2} & \text{if }\ |b|v^m\geq|a|v^{n+m}+|c|.
    \end{cases}
\end{equation*}
Finally, when $|a|v^{n+m}$, $|b|v^m$, and $|c|$ do not form the sides of any triangle, then 

\begin{itemize}
    \item[i)] $\text{Card}(Z_h(v))=0$ for $0<v<\mathfrak{c}$.
    \item[ii)] Including the case $\mathfrak{b}_1=\mathfrak{b}_2$, $\text{Card}(Z_h(v))\leq m$ for $0<v<\mathfrak{b}_1$ and $\text{Card}(Z_h(v))\geq m$ for $v>\mathfrak{b}_2$.
    \item[iii)] $n\leq \text{Card}(Z_h(v))\leq n+3m$ for $v\in(\mathfrak{a},\infty)$. 
\end{itemize}
And therefore,
\begin{equation*}
    \text{Card}(Z_h(v))=
    \begin{cases}
      0 & \text{if }\ |c|>|a|v^{n+m}+|b|v^m, \\
      n+3m & \text{if }\ |a|v^{n+m}>|b|v^m+|c|,\\
      m & \text{if }\ |b|v^m>|a|v^{n+m}+|c|.
    \end{cases}
\end{equation*}
\end{theorem}

\begin{example}
    For $h=z^3+\overline{z}+\sqrt{2}$ and $v=1$, the lengths $1, 1, \sqrt{2}$ form the sides of a non-degenerate triangle. The only integer between $\frac{7}{4}$ and $\frac{9}{4}$ is two, therefore $k=1$.
    Indeed, the roots of $h$ have moduli $|z_1|\approx 0.83404$, $|z_2|=|z_3|\approx 1.35620$. Since $P_*=-2$ and $w_*(v)=1/4$, then $\text{Card}_{\mathbb{Z}}(P_*-w_*(0,v))=\text{Card}_{\mathbb{Z}}(-2.25,-2)=0$, $\text{Card}_{\mathbb{Z}}(P_*+w_*(0,v))=\text{Card}_{\mathbb{Z}}(-2,-1.75)=0$, and $\textbf{1}_{\{P_*\in\mathbb{Z}\}}=1$. That is, there are $k=1$ roots of $h$ whose modulus is less than $v=1$.
    
\end{example}

\subsection{Egerváry Theorems for Harmonic Trinomials}
Below we present the definition and main result for two Egerváry-equivalent harmonic trinomials presented in \cite{barrera2024egervary}. This Theorem \ref{teo:Egervaryarm1} proves useful in some of the proofs in Section \ref{Results}.

\begin{definition}
    Consider the harmonic trinomials
    \begin{equation}
    \label{Fncarm1}
h_1(z)=a_1z^{n+m}+b_1\overline{z}\,^m+c_1 \quad \text{for all } z\in\mathbb{C},
    \end{equation}
    \begin{equation}
    \label{Fncarm2}
h_2(z)=a_2z^{n+m}+b_2\overline{z}\,^m+c_2 \quad \text{for all } z\in\mathbb{C},
    \end{equation}
    whose coefficients are nonzero complex numbers. We say that $h_1$ and $h_2$ are \textbf{Egerváry equivalent}, denoted by $h_1\sim h_2$, if one of the following two conditions holds:
    \begin{enumerate}
        \item[a)] $h_1(z)=kh_2(e^{i\delta}z)$ \quad for all $z\in\mathbb{C}$,
        \item[b)] $h_1(z)=k\overline{h}_2(e^{i\delta}z)$ \quad for all $z\in\mathbb{C}$,
    \end{enumerate}
    where $k$ is a nonzero complex number and $\delta$ a real number, and where $\overline{h}_2(z)=\overline{a}_2z^{n+m}+\overline{b}_2\overline{z}\,^{m}+\overline{c}_2$ for all $z\in\mathbb{C}$.
\end{definition}

\begin{theorem}[Egerváry for harmonic trinomials, \cite{barrera2024egervary}]
\label{teo:Egervaryarm1}
    Let $h_1$ and $h_2$ be harmonic trinomials as in \eqref{Fncarm1} and \eqref{Fncarm2}. We denote by $\alpha_j:=\arg(a_j)$, $\beta_j:=\arg(b_j)$, and $\gamma_j:=\arg(c_j)$ for $j=1,2$. The harmonic trinomials $h_1$ and $h_2$ are equivalent if and only if
    $$
    \frac{|a_1|}{|a_2|}=\frac{|b_1|}{|b_2|}=\frac{|c_1|}{|c_2|},\quad\text{ and }
    $$
    $$
        m(\alpha_1\pm\alpha_2)+(n+m)(\beta_1\pm \beta_2)-(n+2m)(\gamma_1\pm\gamma_2)\equiv 0\pmod{2\pi}.
    $$
\end{theorem}

\begin{example}
    The harmonic trinomials $h_1(z)=z^{5}+3\overline{z}\,^2+2$ and $h_2(z)=2z^{5}-6\overline{z}\,^2-4$ are equivalent since 
    $$
    \frac{|1|}{|2|}=\frac{|3|}{|-6|}=\frac{|2|}{|-4|},\quad\text{ and}
    $$
    $$
        2(0\pm0)+5(0\pm \pi)-7(0\pm\pi)\equiv 0\mod{2\pi}.
    $$
\end{example}

\subsection{Results presented by Thorsten y Wolff for trinomial equations}
In \cite{theobald2016norms}, Thorsten, T. and Wolff, T. presented some advances in the study of the parameter space of the coefficients of trinomials of the form
\begin{equation*}
 f(z)=z^{n+m} + bz^m + c, \,\,\,\,z\in\mathbb{C},   
\end{equation*}
where \( b, c \in \mathbb{C} \) and \( n, m \) are positive coprime integers.

Since our interest focuses on the study of the parameter space of the coefficients of harmonic trinomials, in this subsection we provide a summary of the results presented in Section 4 of \cite{theobald2016norms} in order to identify the differences and similarities between the results established for trinomials and those obtained when considering harmonic trinomials.
\begin{itemize}
    \item First, they show that, for a fixed \( c \in \mathbb{C}^* \), the geometric structure of the set of all \( b \in \mathbb{C} \) such that \( f \) has a root of norm \( v \) corresponds to a hypotrochoid, where
\begin{equation*}
\label{param_b_arm}
   -b\in\Big\{(R-r)\cdot e^{i\cdot\varphi}+d\cdot e^{i\big(\gamma+\frac{(R-r)}{2R}\cdot\varphi\big)}:\varphi\in[0,2\pi)\Big\}
\end{equation*}
and 
$R=\frac{v^n}{m}(n+m)$, $r=\frac{v^n}{m}\cdot n$ and $d=|c|\cdot v^{-m}$. They also show that, for a fixed \( b \in \mathbb{C}^* \), the geometric structure of the set of all \( c \in \mathbb{C} \) such that \( f \) has a root of norm \( v \) corresponds to an epitrochoid (see Theorems 4.1, 4.16 in \cite{theobald2016norms}).

\item Then, they show that, for a trinomial \( f \), there exist at most two roots with norm \( v \). Moreover, two roots of \( f \) have the same norm \( v \) if and only if \( b \) lies on a union of a set of rays. This set consists of several rays that all start at the origin in the complex plane and extend outward in specific directions determined by the parameters $n,m\in\mathbb{N}^*$ and $\gamma=\arg c$. Each ray corresponds to one of the possible angles $\theta_k=\frac{(n+2m)\gamma+k\pi}{n+m}$ for $k=0,1,\cdots,2(n+m)-1$.
Which are described as even or odd rays according to whether \( k \) is even or odd.
In the same sense, they also prove that there exist two roots of \( f \) with the same norm if and only if \( b \) is a singular point of a hypotrochoid (see Proposition 4.3, Theorems 4.4, 4.5 and Corollary 4.6 in \cite{theobald2016norms}).

\item They remark that for trinomials with nonzero real coefficients, and after ordering all the roots according to their norm, the real roots of \( f \) occur in the positions \( 1, m, m+1 \), or \( n+m \) (see Theorem 4.8 in \cite{theobald2016norms}).

\item They also define a set \( U_j^A \) as the set of all trinomial equations \( f \) whose \( j \)-th and \( (j+1) \)-th roots (ordered by their norms) have different norm, for \( 1 \leq j \leq n+m-1 \). Subsecuently, this is used to show that, for a fixed \( c \in \mathbb{C}^* \) and given \( f_b = z^{n+m} + bz^m + c \), a parametric family of trinomials with parameter \( b \in \mathbb{C} \), the following holds: if \( n+j \) is even (respectively odd), \( f_b \) belongs to \( U_j^A \) if and only if \( b \) does not belong to the union of rays in even (respectively odd) position. Moreover, by using the discriminant defined by Greenfield and Drucker in \cite{GREENFIELD1984105}, they prove that any trinomial on the hypersurface defined by this discriminant lies on the boundary of \( (U_m^A)^c \). Finally, they show that for a fixed \( c \in \mathbb{C}^* \), if \( n+j \) is even (respectively odd), \( f \) belongs to the set \( U_m^A \) if and only if \( b \) does not belong to a closed disk of radius \( r \) intersected with the union of rays in even (respectively odd) position (see Theorem 4.9 and Corollaries 4.12, 4.13 in \cite{theobald2016norms}).
\end{itemize}

\subsection{Trochoids}
\label{Trochoids}
Section \ref{Results} presents the parametrization of certain harmonic trinomials using trochoid curves. Now, we review the definitions corresponding to these parametric curves.

\begin{definition}
    A \textit{trochoid} is described in \cite{lawrence2013catalog} and \cite{MONSINGH2024100928} as the set of points in the plane generated by a fixed point on the radius of a circle (or an extension of it) that rolls along a fixed straight line (or a fixed circle) without slipping.

    The general parametrization of trochoids is given by:
\begin{equation}
\label{Param_gnral_trocoide}
\begin{matrix}
    u(\varphi) & = & u_0+a\cos(x_1 \varphi+\eta)+d\cos(x_2 \varphi+\gamma),\\
    v(\varphi) & = & v_0+a\sin(x_1 \varphi+\eta)+d\sin(x_2 \varphi+\gamma),
\end{matrix}
\end{equation}
where $a$, $d$, $x_1$, $x_2$, $\eta$, and $\gamma$ are constants, and $\varphi$ is the independent variable.  
The center of the trochoid is given by $(u_0,v_0)$. For simplicity, when this center is the origin, we will refer to it simply as a \textit{trochoid}.
\end{definition}

We will be interested in the parametrization of a trochoid as in \eqref{Param_gnral_trocoide}, centered at the origin, with parameters $R, r, d \in \mathbb{R}^+$, where $a = R - r$, $\eta=0$, $x_1=1$ y $x_2 = \frac{R - r}{2R}$:
\begin{equation*}
    \begin{matrix}
        u(t) & = & (R - r)\cos(t) + d\cos\bigg(\frac{R - r}{2R}t + \gamma\bigg),\\[0.2cm]
        v(t) & = & (R - r)\sin(t) + d\sin\bigg(\frac{R - r}{2R}t + \gamma\bigg).
    \end{matrix}
\end{equation*}
Whose complex equation takes the form:
\begin{equation}
\label{ComplexForm}
  T(\varphi)= (R-r)e^{i\varphi}+de^{i\big(\gamma+\frac{R-r}{2R}\varphi\big)}.
\end{equation}

\subsection{Jacobian of harmonic functions}
Since it is not possible to extend the discriminant result presented by Greenfield and Drucker in \cite{GREENFIELD1984105} to harmonic trinomials with nonreal complex variable, we now provide definitions, remarks, and computations that will later be used to address the interpretation of the discriminant for harmonic trinomials.
\begin{definition}
    A harmonic polynomial is a harmonic function (that is, it satisfies Laplace’s equation) of the form $H(z) = P(z) + \overline{Q(z)}$,
where \( P \) and \( Q \) are complex polynomials and \( z = x + iy \) with real variables \( x \) and \( y \).
\end{definition}
A complex function (such as \( P \) and \( Q \)) is analytic in a region if it satisfies the Cauchy–Riemann equations. Moreover, if a function is analytic, then its real and imaginary parts are harmonic, and hence \( H(z) \) is harmonic \cite{Duren_2004}.

Taking into account the previous observations, let \( P(z) = P_1(x, y) + iP_2(x, y) \) and \( Q(z) = Q_1(x, y) + iQ_2(x, y) \), where \( P \) and \( Q \) are analytic functions satisfying the Cauchy–Riemann equations. Then, we have
$P'(z)  = \frac{\partial P_1}{\partial x} + i\frac{\partial P_2}{\partial x} = \frac{\partial P_1}{\partial x} - i\frac{\partial P_1}{\partial y}, \quad$ y $
Q'(z) = \frac{\partial Q_1}{\partial x} + i\frac{\partial Q_2}{\partial x} = \frac{\partial Q_1}{\partial x} - i\frac{\partial Q_1}{\partial y}.$
Hence, the Jacobian of \(H(z)\) is given by
\begin{equation}
\label{Jacobiano}
    J_H(z)\;=\;|P'(z)|^2-|Q'(z)|^2.
\end{equation}

In \cite{Duren_2004}, \cite{Duren01051996} the \textit{order} of a zero $z_0$ of a harmonic function $H(z) = P(z) + \overline{Q(z)}$ is defined when the Jacobian as in \eqref{Jacobiano}, is nonzero, and this definition it given in terms of Taylor's expansions:
\begin{equation*}
    P(z)=p_0+\sum_{k=1}^{\infty}p_k(z-z_0)^k,\qquad Q(z)=q_0+\sum_{k=1}^{\infty}q_k(z-z_0)^k.
\end{equation*}
If $H$ sense-preserving, $z_0$ has the order of the series $\sum_{k=1}^{\infty}p_k(z-z_0)^k$ and if $H$ sense-reversing, $z_0$ has the order of the series $\sum_{k=1}^{\infty}q_k(z-z_0)^k$.

When $J_H(z_0)=0$, the harmonic trinomial equation has zeros of multiplicity at least two, which does not necessarily correspond to the given definition of order de $z_0$.

Moreover, following the ideas presented in \cite{S_te_2021}, the critical curve of $H$, defined as the set of all points in $\mathbb{C}$ where the Jacobian vanishes, corresponds, in the case of harmonic trinomials, to a circle centered at the origin of a specific radius.

For example, $z_0=1$ is a zero of multiplicity two for $H=z^2-2\overline{z}+1$, where $J_H=0$ at $z_0=1$. Rewriting $H$ as $H=(z-1)^2+2(z-1)-\overline{2(z-1)}$, we observe that $z_0=1$ is a zero of ``order" 1.

\section{Results}
\label{Results}

In this section, we describe the space $H$ of all harmonic trinomials defined as in \eqref{fnc}, identify the geometry of harmonic trinomials having one or two roots with the same norm, and study some geometric and topological properties of the subsets $U_j$ and their complements defined in the Introduction.

\begin{theorem}
\label{teo:param_trocoide}
    Let $h(z)=z^{n+m}+b\,\overline{z}\,^m+c$ with $b\in\mathbb{C}$ and $c\in\mathbb{C}^*$, be a harmonic trinomial, and let $v\in\mathbb{R}^+$. Then $h$ has a root of norm $v$ if and only if $b$ lies on a trochoid (up to rotation) with parameters $R=\big(\frac{v^n}{2m}\big)(n+2m)$, $r=\frac{v^n}{2m}\cdot n$, and $d=\frac{|c|}{v^m}$.
\end{theorem}

\begin{remark}
    The following proof is an adaptation of the proof of Theorem 4.1 in \cite{theobald2016norms}, expressed now in the context of harmonic trinomial equations.
\end{remark}

\begin{proof}
Let $\zeta = v e^{i\theta}$ be a root of $h$ and write $c = |c| e^{i\gamma}$, then $- b = v^{n} e^{i(n+2m)\theta} + \frac{|c|}{v^{m}} e^{i(\gamma + m\theta)}$ . Using $R-r=v^n$, $\dfrac{R-r}{2R}=\dfrac{m}{n+2m}$, and the change of variable $\varphi=(n+2m)\theta$, we obtain $m\theta=\Big(\dfrac{R-r}{2R}\Big)\varphi$  and hence
\begin{equation}
\label{trochoid}
    -b\in\Big\{(R-r)e^{i\varphi}+de^{i\big(\gamma+\frac{R-r}{2R}\varphi\big)} : \varphi\in[0,2\pi)\Big\}.
\end{equation}
Which corresponds to the complex form of the parametrization of a trochoid centered at the origin, as in \eqref{ComplexForm}, defined in Subsection \ref{Trochoids}.
\end{proof}

\begin{example}
\label{example_hip}
    The harmonic trinomial $h_1(z)=z^8+b\,\overline{z}\,^3+\frac{1}{2}$, $h_2(z)=z^7+b\,\overline{z}\,^2+2$ and $h_3(z)=z^5+b\,\overline{z}+1$ have, respectively, a root of norm one if and only if $b\in\mathbb{C}$ is located on the trochoid with parameters $(R,r,d)=(11/6,5/6,1/2)$, $(R,r,d)=(9/4,5/4,2)$ and $(R,r,d)=(3,2,1)$, respectively. See Figures \ref{fig:1}, \ref{fig:2} and \ref{fig:3}.\\
The parameterization of $-b$ in a harmonic trinomial equation in the Theorem  \eqref{teo:param_trocoide} to a \textit{trochoid} details similarities to a \textit{hypotrochoid} as presented in \cite{theobald2016norms}. The presence of $\big(\frac{R-r}{2R}\big)$ in the second term of \eqref{trochoid} instead of $\big(\frac{r-R}{r}\big)$ in the second term of the parameterization of a standard hypotrochoid changes the periodicity and behavior of this curve.\\

\begin{figure}[h!] 
    \centering 
    \begin{subfigure}[b]{0.3\textwidth} 
        \centering
        \includegraphics[width=\textwidth]{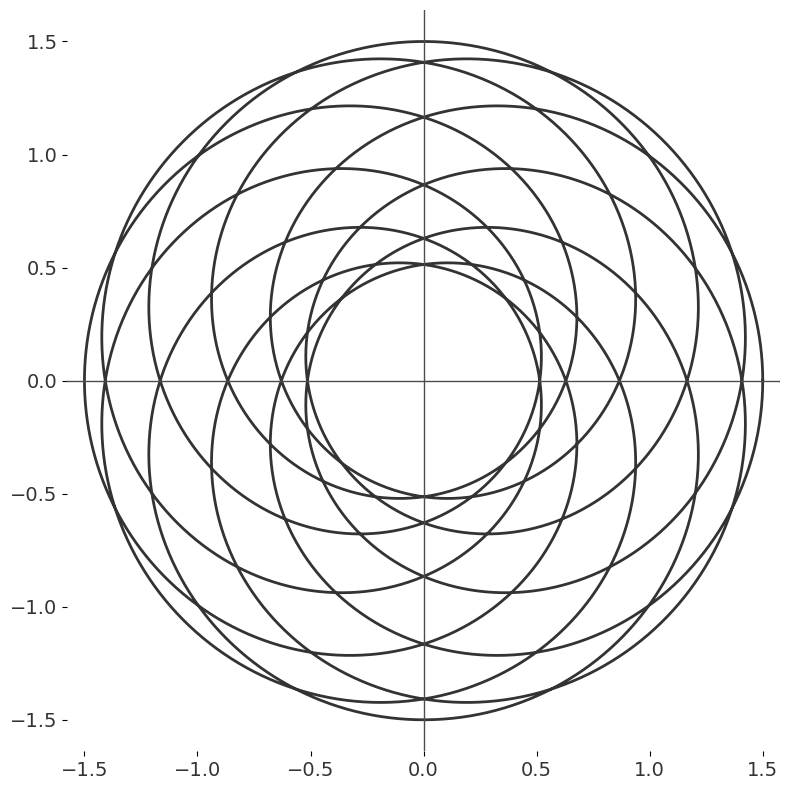} 
        \caption{\small{Trochoid with parameters $(R,r,d)=(11/6,5/6,1/2)$.}}
        \label{fig:1}
    \end{subfigure}
    \hfill 
    \begin{subfigure}[b]{0.3\textwidth}
        \centering
        \includegraphics[width=\textwidth]{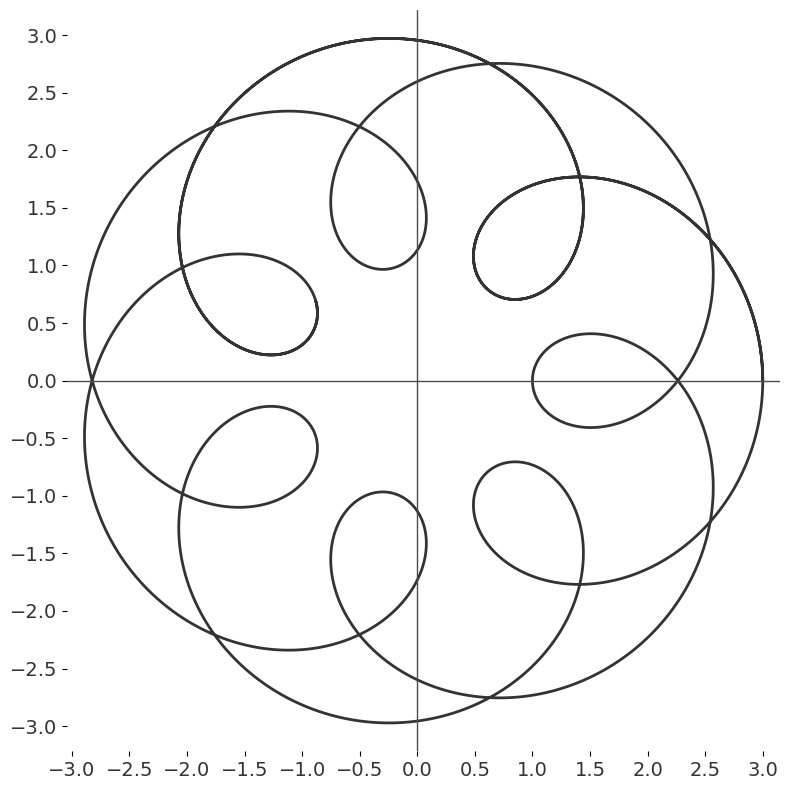}
        \caption{\small{Trochoid with parameters $(R,r,d)=(9/4,5/4,2)$.}}
        \label{fig:2}
    \end{subfigure}
    \hfill
    \begin{subfigure}[b]{0.3\textwidth}
        \centering
    \includegraphics[width=\textwidth]{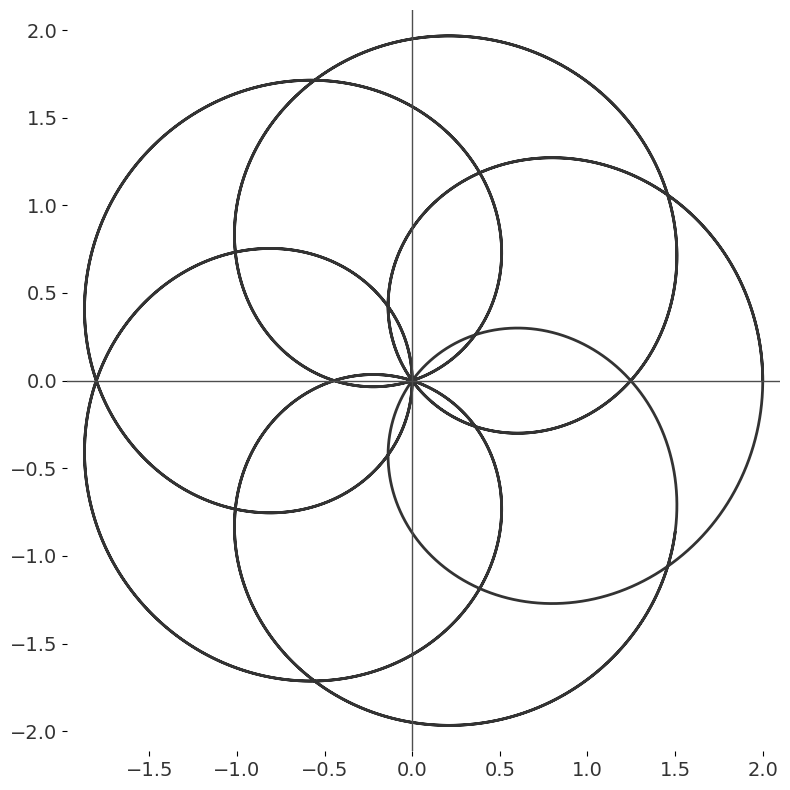}
        \caption{\small{Trochoid with parameters $(R,r,d)=(3,2,1)$.}}
        \label{fig:3}
    \end{subfigure}
    \caption{\small{Parametrization of $h_1$, $h_2$ y $h_3$ with a root of norm $v=1$ to trochoids with parameters $R$, $r$ y $d$.}} 
    \label{fig:multiple}
\end{figure}
\end{example}

\begin{theorem}
\label{teo:param_trocoide2}
Let \( h(z) = z^{n+m} + b\overline{z}\,^m + c \) with \( b \in \mathbb{C} \) and \( c \in \mathbb{C}^* \) be a harmonic trinomial, and let \( v \in \mathbb{R}^+ \). Then \( h \) has a root of norm \( v \) if and only if \( c \) is positioned, up to rotation, on a trochoid with the following parameters $R = \frac{v^m \cdot |b| \cdot n}{n+m}$, $r = \frac{v^m \cdot |b| \cdot m}{n+m}$, and 
$d = v^{n+m}$.
\end{theorem}
\begin{proof}
    The proof is analogous to that of Theorem \eqref{teo:param_trocoide}.
\end{proof}

\begin{example}
    As a canonical counterpart to the examples in \eqref{example_hip}, the harmonic trinomials 
$h_4 = z^8 + \frac{3}{2}\overline{z}\,^3 + c, 
h_5 = z^7 - \frac{7}{2}\overline{z}\,^2 + c \text{\,\,\,and\,\,\,} 
h_6 = z^5 + \overline{z} + c$ have a root of norm one if and only if \( c \) is positioned on the trochoids with parameters $(R,r,d)=(15/16,9/16,1), (5/2,1,1)$ and $(4/5,1/5,1)$,
respectively. See Figures \ref{fig:7}, \ref{fig:8} and \ref{fig:9}.
\begin{figure}[h!] 
    \centering 
    \begin{subfigure}[b]{0.3\textwidth} 
        \centering
        \includegraphics[width=\textwidth]{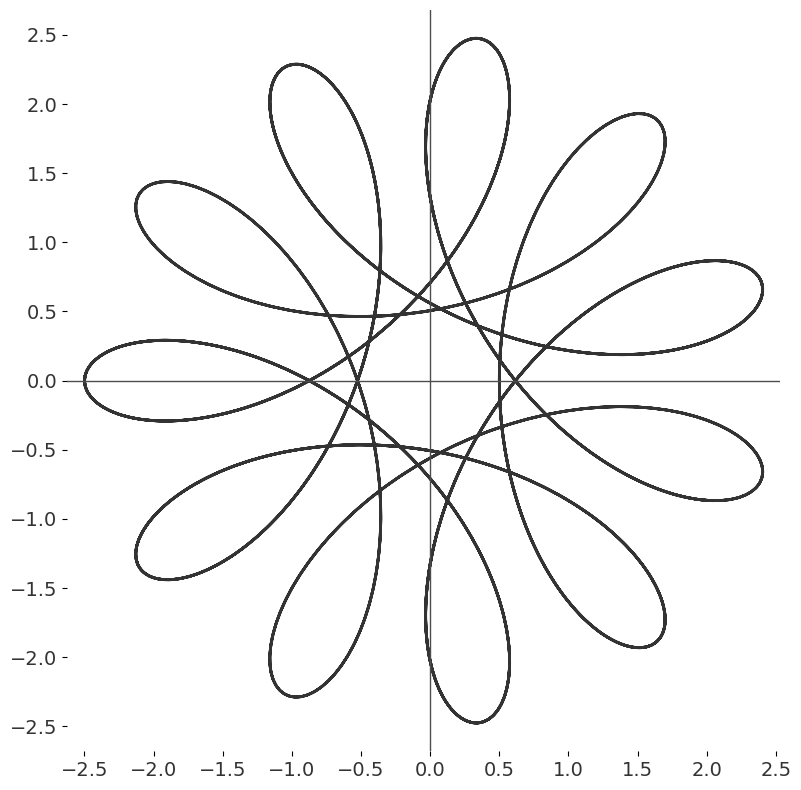} 
        \caption{\small{Trochoid with parameters $(R,r,d)=(15/16,9/16,1)$.}}
        \label{fig:7}
    \end{subfigure}
    \hfill 
    \begin{subfigure}[b]{0.3\textwidth}
        \centering
        \includegraphics[width=\textwidth]{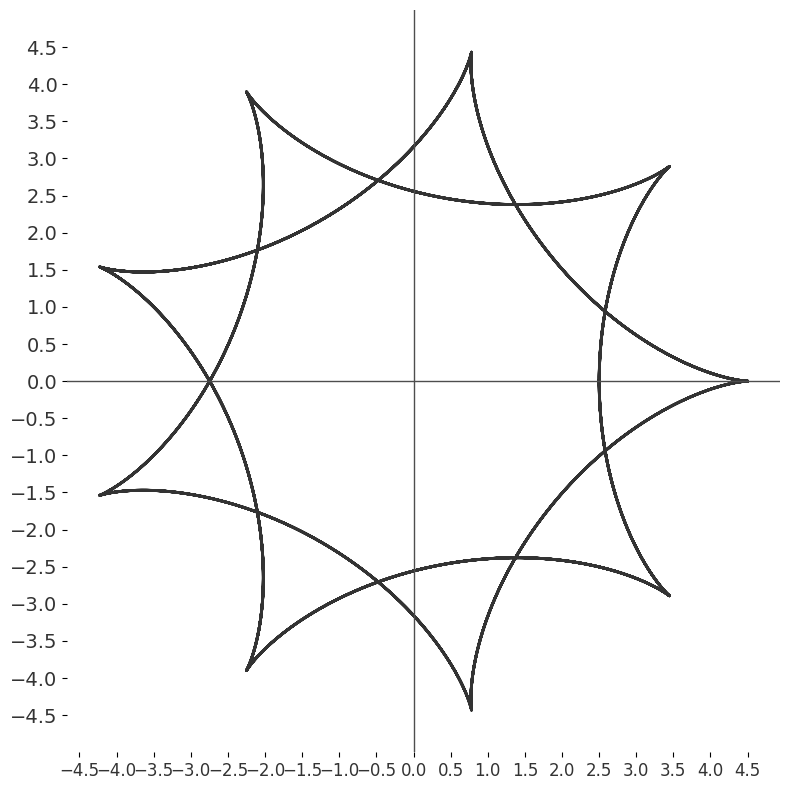}
        \caption{\small{Trochoid with parameters $(R,r,d)=(5/2,1,1)$.}}
        \label{fig:8}
    \end{subfigure}
    \hfill
    \begin{subfigure}[b]{0.3\textwidth}
        \centering
        \includegraphics[width=\textwidth]{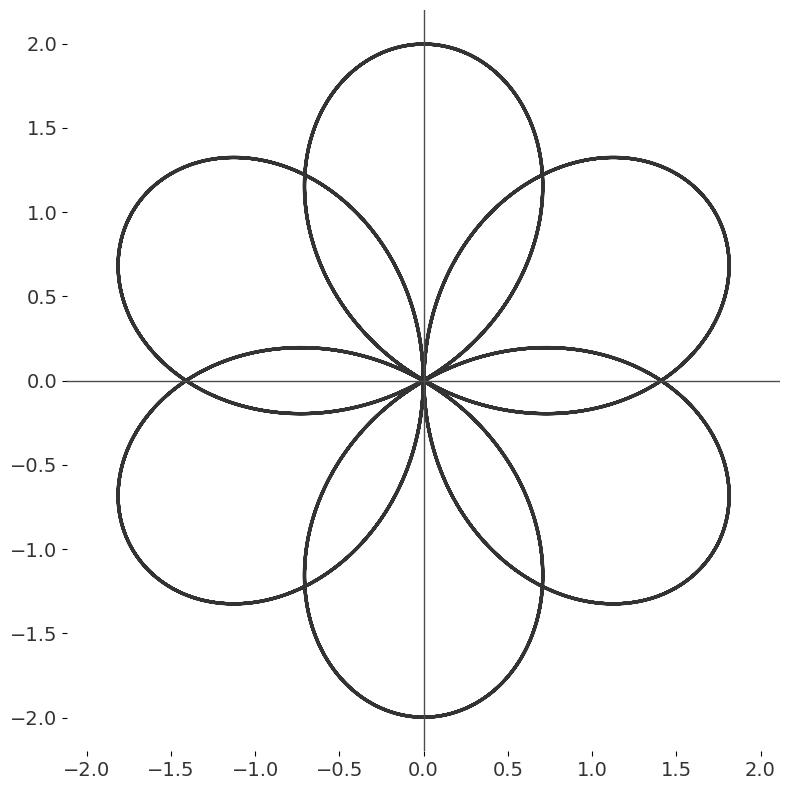}
        \caption{\small{Trochoid with parameters $(R,r,d)=(4/5,1/5,1)$.}}
        \label{fig:9}
    \end{subfigure}
    \caption{\small{Parametrization of $h_4$, $h_5$ y $h_6$ with a root of $v=1$ to trochoids with parameters $R$, $r$ y $d$.}} 
    \label{fig:multiple}
\end{figure}
\end{example}

\begin{proposition}
\label{raicesdeh}
    Let $h(z)=z^{n+m}+b\,\overline{z}\,^m+c$, with $b,c\in\mathbb{C}^*$, $n$ and $m$ coprime integers, and 
    $v\in\mathbb{R}^+$. Let $\zeta$ be a root of $h$ such that $|\zeta|=v$. Then $h$ has at most two roots of norm $v$.
\end{proposition}
\begin{proof}
   Let $h(z)=z^{n+m}+b\,\overline{z}\,^m+c$ with $b\in\mathbb{C}$ and $c\in\mathbb{C}^*$. Let $\zeta$ be a root of $h$ such that $|\zeta|=v\in\mathbb{R}^+$. We note that $\zeta\overline{\zeta}=v^2$, then $h(\zeta)=0$ if and only if $\zeta^{n+2m}+b\,v^{2m}+c\,\zeta^m  =0$. Define $p(z):=z^{n+2m}+cz^m+b\,v^{2m}$, where $p(\zeta)=0$. We note that $n,m$ are coprimes iff  $n+2m$ and $m$ are coprime.  
   Moreover, by Proposition~4.3 in \cite{theobald2016norms}, the polynomial $p(z)$ has at most two roots of modulus $v$.  
   Therefore, the harmonic trinomial equation $h$ has at most two roots of modulus $v$.
\end{proof}

\begin{example}
The harmonic trinomial $h=z^3-\overline{z}+(\frac{1}{3})^{3/2}$ has the following roots:

$z_1\approx 0.2005$, $z_2\approx 0.8846$, $z_3\approx 0.0479+1.0034 i$, $z_4\approx 0.0479-1.0034 i$ and $z_5\approx -1.0851$.

Where $|z_3|=|z_4|\approx1.0046$. 
\end{example}

In this context, and following the ideas of Thorsten and Wolff in \cite{theobald2016norms}, it is defined the union of rays as
\begin{equation}
\label{rayos}
    \mathcal{R}(n,m,c)=\bigcup_{0\leq k\leq 2(n+m)-1}\mathbb{R}^+\cdot e^{i\left(\frac{(n+2m)\gamma+k\pi}{n+m}\right)},
\end{equation}
for the parameters $n,m\in\mathbb{N}^*$ and $c\in\mathbb{C}^*$, given by a harmonic trinomial $h(z)=z^{n+m}+b\,\overline{z}\,^m+c$.

The $2(n+m)$ rays defined by $\mathcal{R}(n,m,c)$ in \eqref{rayos} are described as even rays ($\mathcal{R}^{\text{even}}$) or odd rays ($\mathcal{R}^{\text{odd}}$), depending on whether $k$ is even or odd, where $0\leq k\leq 2(n+m)-1$.


\begin{theorem}
\label{teo:union_rayos}
    Let $h(z)=z^{n+m}+b\,\overline{z}\,^m+c$, with $b\in\mathbb{C}$ and $c\in\mathbb{C}^*$, such that for two roots $z_1, z_2$ of $h$ we have $|z_1|=|z_2|$. Then $((n+m)\beta-(n+2m)\gamma)/\pi\in\mathbb{Z}$ iff $b\in \mathcal{R}(n,m,c)$.  
    In particular, $(U_j)^c\subseteq \mathcal{R}(n,m,c)$ for each $1\leq j\leq n+m-1$.
\end{theorem}

 \begin{remark}
     The analogous result presented in Theorem 4.4 in \cite{theobald2016norms} corresponds to only one of the inclusions of the “if and only if” statement. However, Theorem \ref{teo:union_rayos} proves the double inclusion (which also holds for non-harmonic trinomials).
 \end{remark}

\begin{proof}
    Let $z_1, z_2$ be roots of $h$ be such that $|z_1|=|z_2|=v\in\mathbb{R}^+$. By Corollary 1.5 in \cite{barrera2024egervary}, we have that $h$ is \textit{Egerváry equivalent} to a harmonic trinomial with real coefficients. Moreover, by Corollary 1.4(ii) in \cite{barrera2024egervary}, the relation $m\alpha+(n+m)\beta-(n+2m)\gamma\equiv 0\pmod{\pi}$ holds. Recall that we assumed $\alpha=0$, hence $\frac{(n+m)\beta-(n+2m)\gamma}{\pi}\in\mathbb{Z}$.
    Now, $b\in\mathcal{R}(n,m,c)$ follows from the Definition \eqref{rayos} of the union of rays, and $(U_j)^c\subset\mathcal{R}(n,m,c)$ follows from the definition of $U_j$.

    On the other hand, let $b=\lambda\cdot e^{i\left(\frac{(n+2m)\gamma+k\pi}{n+m}\right)}$, with $\lambda\in\mathbb{R}^+$. Then $\arg(b):=\beta=\frac{(n+2m)\gamma+k\pi}{n+m}$ for some $0\leq k\leq 2(n+m)-1$.  
Hence, $(n+m)\beta-(n+2m)\gamma=k\pi$ for some $0\leq k\leq 2(n+m)-1$.
\end{proof}

\begin{example}
\label{EjemArmPtob}
    For $h = z^{2} - 2\,\overline{z} + 1$, the roots $-1+2i$ and $-1-2i$ have the same modulus. Observe that 
    $
    \frac{(n+m)\beta - (n+2m)\gamma}{\pi}
    = 2 \in \mathbb{Z}
    $
    and therefore $b = -2 \in \mathcal{R}(n,m,c)$. In particular, $(U_j^A)^c \subseteq \mathcal{R}(n,m,c)$ for $j=1$ (see Figure \ref{fig:ArmGrado2}).

\begin{figure}[h!]
\centering

\begin{minipage}{0.35\textwidth}
    \includegraphics[width=\linewidth]{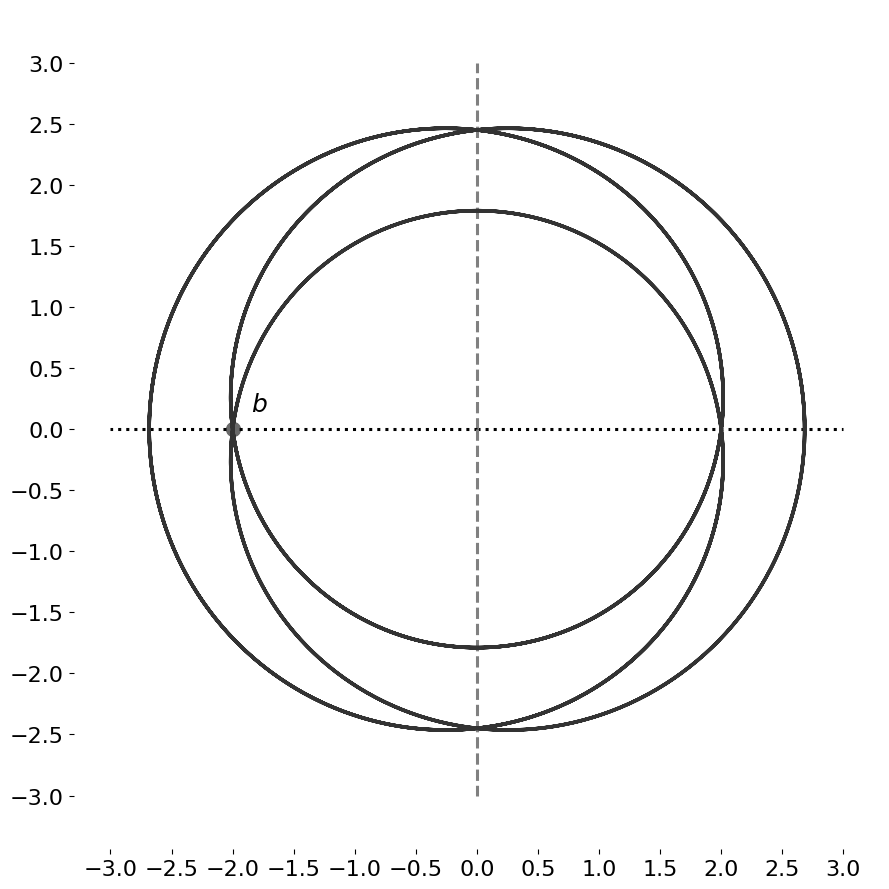}
\end{minipage}
\hfill
\begin{minipage}{0.64\textwidth}
    \caption{\small{Trochoid associated to $h=z^{2}-2\overline{z}+1$ taking a root with modulus $\sqrt{5}$. The dotted lines represent the even rays ($\mathcal{R}^{\mathrm{even}}$) and the dashed lines represent the odd rays ($\mathcal{R}^{\mathrm{odd}}$). We observe that $b=-2\in\mathcal{R}$.}}
    \label{fig:ArmGrado2}
\end{minipage}
\end{figure}

\end{example}

\begin{theorem}
\label{RaizDoble}
    Let $h(z)=z^{n+m}+b\,\overline{z}^{\,m}+c$ with $b\in\mathbb{C}$ and $c\in\mathbb{C}^*$. There exist $z_j, z_{j+1}$ roots of $h$ with $|z_j|=|z_{j+1}|=v\in\mathbb{R}^+$ if and only if $b$ is a singular point of the trochoid $h^v - b$ determined in \eqref{trochoid}. In particular:
    \begin{enumerate}
        \item $h$ has two distinct roots with the same modulus $v$ if and only if $b$ lies on a double point of a trochoid.
        \item $h$ has a root of multiplicity two with modulus $v$ if and only if $b$ is a cusp and $b=\frac{n+m}{m}v^n e^{i(n+2m)\phi}$.
        \item For $n>m$, $h$ has more than two roots with modulus $v$ if and only if $b=0$, which holds if and only if the trochoid is a rhodonea curve with a point of multiplicity $n+m$ at the origin.
    \end{enumerate}
\end{theorem}

\begin{remark}
    The following proof is an adaptation of the proof of Theorem 4.5 in \cite{theobald2016norms}, now formulated for harmonic trinomials and employing a standard differential–geometric argument to derive a system of equations.
\end{remark}

\begin{proof}
    \begin{enumerate}
        \item Let $z_j$ and $z_{j+1}$ be two distinct roots of $h$ with the same modulus $v \in \mathbb{R}^+$. The condition $z_j \neq z_{j+1}$ is equivalent to the existence of two distinct angles 
$\phi, \varphi \in [0,2\pi)$, with $\phi \neq \varphi$ and $h^v(\phi)=h^v(\varphi)=0$, where
        \begin{equation*}
            h^v(\phi)=b+v^n e^{i(n+2m)\phi}+|c|v^{-m}e^{i(\gamma+m\phi)}=0,\qquad \text{for some }\phi\in[0,2\pi).
        \end{equation*}

        Thus $h^v(\phi)-b=h^v(\varphi)-b=-b$. This happens if and only if the trochoid attains the value of the parameter $-b$ twice. That is, $-b$ lies on a double point of the trochoid.

        \item
        We now suppose that $z_j = z_{j+1}$, that is, $h$ has a root of multiplicity $2$. We consider the limit of a family of harmonic trinomials given by $\phi \to \varphi$ and, consequently, $z_j \to z_{j+1}$.
 This implies that the intersection point of two smooth branches of the curve $h^v(\phi)-b$ degenerates into a cusp.

        Now, if the trochoid $h^v(\phi)-b$ has a cusp, then
        \begin{equation}
        \label{ec1}
            h^v(\phi)=b+v^n e^{i(n+2m)\phi}+|c|v^{-m}e^{i(\gamma+m\phi)}=0,\quad\text{ and }
        \end{equation}
        \begin{equation}
        \label{ec2}
               \frac{\partial}{\partial\phi}h^v(\phi) = v^ni(n+2m)e^{i(n+2m)\phi}
               +|c|v^{-m}im e^{i(\gamma+m\phi)} = 0.
        \end{equation}

        Solving the system given by \eqref{ec1} and 
        \eqref{ec2} it follows that $b=\frac{n+m}{m}\, v^n e^{i(n+2m)\phi},$ and hence
        \begin{equation*}
            |b|e^{i\beta}=\frac{n+m}{m}v^n e^{i(n+2m)\phi}.
        \end{equation*}
        Therefore $|b|=\frac{n+m}{m}v^n$ and $\beta\equiv(n+2m)\phi\pmod{\pi}$, so that $\phi=\frac{\beta}{n+2m}$.
        
We observe that, for $\zeta=ve^{i\phi}$ a root of $h$, we have $h(\zeta)=v^{n+m}e^{i(n+m)\phi}+b e^{-im\phi}+|c|e^{i\gamma}=0,$ thus $ \frac{\partial}{\partial\phi}h(\zeta)= 0.$ 
        That is, $h$ vanishes at $\zeta=ve^{i\phi}$ and the derivative of $h$ with respect to $\phi$ is also zero at $ve^{i\phi}$. Hence, $h$ has a root of multiplicity $2$ and modulus $v$.
        \item
        We suppose that $h$ has more than two roots of modulus $v$. By Proposition \ref{raicesdeh}, this implies that $b=0$ or $c=0$, but by hypothesis $c\in\mathbb{C}^*$, so we must have $b=0$. If $b=0$, then $h(z)=z^{n+m}+c$. This means that the trochoid satisfies the features of a trinomial equation of the form $f(z)=z^{n+m}+bz^m+c$ with $b=0$ (see Theorems 4.1 and 4.2 in \cite{theobald2016norms}).
    \end{enumerate}
\end{proof}

\begin{example}
        Let $h(z)=z^{n+m}+b\,\overline{z}\,^m+c$ with $b\in\mathbb{C}$ and $c\in\mathbb{C}^*$; we observe that:
        \begin{enumerate}
            \item For $h(z)=z^2-2\overline{z}+1$, as in Example \eqref{EjemArmPtob}, $-1+2i$ and $-1-2i$ are two distinct roots with identical modulus $v=\sqrt{5}$, and $b$ lies on a double point of a trochoid (see Figure \eqref{fig:ArmGrado2}).
            \item For $h(z)=z^2+\frac{2\sqrt{3}}{3}\overline{z}-1$, $\frac{\sqrt{3}}{3}$ is a double root with the same modulus, and $b$ is a cusp on the associated trochoid (see Figure \ref{fig:ArmCuspide}).

\begin{figure}[h!]
\centering

\begin{minipage}{0.35\textwidth}
    \includegraphics[width=\linewidth]{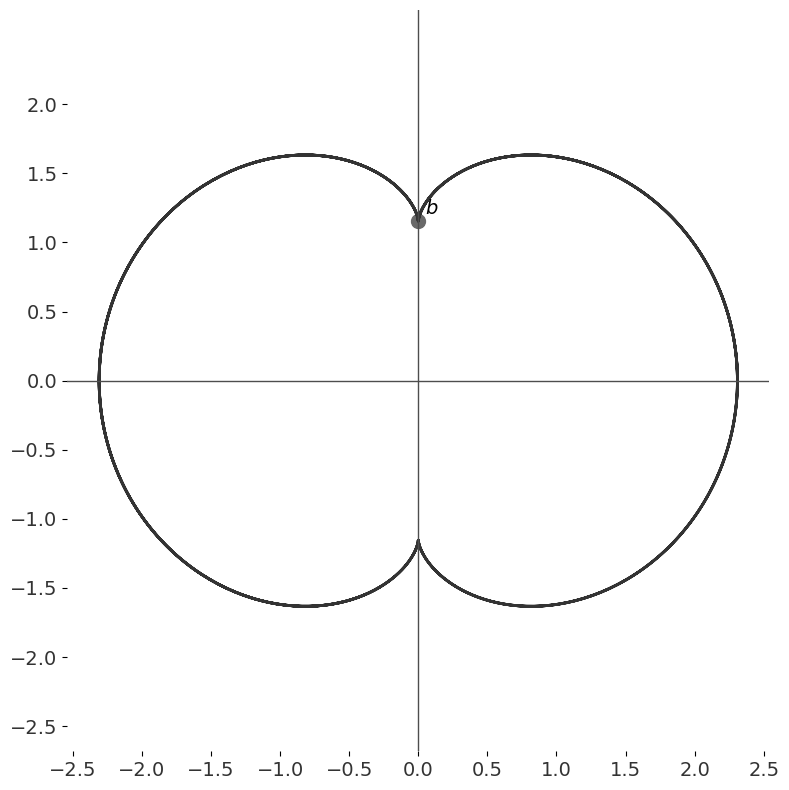}
\end{minipage}
\hfill
\begin{minipage}{0.64\textwidth}
    \caption{\small{Trochoid associated to $h=z^{2}+\frac{2\sqrt{3}}{3}\,\overline{z}-1$ having a double root of modulus $\frac{\sqrt{3}}{3}$.}}
    \label{fig:ArmCuspide}
\end{minipage}

\end{figure}
            
\item If $h=z^3+1$, that is, taking $b=0$, the roots are $z_1=\frac{1}{2}-\frac{\sqrt{3}}{2}i$, $z_2=\frac{1}{2}+\frac{\sqrt{3}}{2}i$, and $z_3=-1$. In other words, $h$ has three roots of modulus $v=1$.
        \end{enumerate} 
    \end{example}

\begin{lemma}
\label{OrderingRoots}
    Let $h(z)=z^{n+m}+b\overline{z}\,^m+c$, with $b,c\in\mathbb{R^*}$, and let $Z(h)=\{z_1,z_2,\ldots,z_{k}\}$ be the complete set of roots of $h$ such that $|z_1|\leq|z_2|\leq\cdots\leq|z_{k}|$, where $k$ denotes the last root of $h$ in the given order and satisfies $k\leq n+3m$. Let $\zeta\in\mathbb{C^*}$ be a root of $h(z)$. Then $-\zeta$ is a root of $h(-z)$. Moreover, the order of the moduli is preserved under the change of variable $-z$.
\end{lemma}

\begin{remark}
    The previous lemma provides a result for harmonic trinomials obtained through a change of variable, which has no analogue for the non-harmonic trinomials presented in \cite{theobald2016norms}.
\end{remark}

\begin{proof}
    We observe that $h(z)=z^{n+m}+b\overline{z}\,^m+c$, with $b,c\in\mathbb{R^*}$, and $h(-z)=(-1)^{n+m}z^{n+m}+(-1)^mb\overline{z}\,^m+c:=g(z)$ are Egervary equivalent. Indeed, $g(z)=kh(e^{i\delta}z)$ for $k=1$, $\delta=\pi$, and for every $z\in\mathbb{C}$. 

    Let $\zeta\in\mathbb{C}^*$ a root of $h(z)$. Since $h(\zeta)=0$, it follows that $g(-\zeta)=h(\zeta)=0$; consequently, $-\zeta$ is a root of $g(z)$. In particular, since $|\zeta|=|-\zeta|$, the order of the moduli is preserved. Moreover, the identity $g(z)=h(e^{i\pi}z)$ can be interpreted geometrically as stating that the roots of $h(z)$ and $g(z)=h(-z)$ differ only by a rotation, but not in their ordering.
\end{proof}

The following result is a reformulation of Theorem 4.8 in \cite{theobald2016norms}, now for harmonic trinomials, whose proof considers the polynomials presented in \eqref{polinomios}, which were not addressed for non-harmonic trinomials.

\begin{theorem}
\label{RaicesReales}
    Let $h(z)=z^{n+m}+b\overline{z}\,^m+c$, with $b,c\in\mathbb{R^*}$, and let $Z(h)$ as in Lemma \ref{OrderingRoots}. We suppose that there exists $\lambda\in\mathbb{R^*}$ root of $h$. Then $\lambda=z_1$, or $\lambda=z_{m-1}$, or $\lambda=z_m$, or $\lambda=z_{m+1}$, or $\lambda=z_{n+m-1}$. 

    Moreover, if $\lambda=z_m$, then the interval $(|z_{m}|,|z_{m+1}|)$ lies outside the region determined by the triangles whose side lengths are $\lambda^{n+m}$, $|b|\,\lambda^m$, and $|c|$. In particular:
    \begin{itemize}
        \item if $\lambda=z_{m-1}=\mathfrak{b_1}$ is a root of $h$, then $z_{m}=\mathfrak{b_2}$,
        \item if $\lambda=z_{m}=\mathfrak{b_1}$ is a root of $h$, then $z_{m+1}=\mathfrak{b_2}$,
        \item if $\lambda=z_{m}=\mathfrak{b}$ is a root of $h$, then $z_{m+1}=\mathfrak{b}$. That is, $\lambda=\mathfrak{b}$ is a double root of $h$.
    \end{itemize}
\end{theorem}

\begin{proof}
    Let $\lambda=z_j\in\mathbb{R}^+$ be a root of $h$.  
    In order for $h(\lambda)$ to vanish, one of the following sign configurations must occur:
    \begin{enumerate}
        \item \textbf{First case:} $b>0$ and $c<0$.  
    Let $C(\lambda)=-\lambda^{\,n+m}-|b|r^m+|c|=0$. Then $|c|=\lambda^{\,n+m}+|b|\lambda^m$, and consequently $-c=\lambda^{\,n+m}+b\lambda^m$.  
    By Theorem~2.8 in \cite{barrera2022number}, it follows that $\lambda=\mathfrak{c}$.  
    Since $h$ has no roots of modulus smaller than $\lambda=\mathfrak{c}$, we conclude that $\lambda=z_1=\mathfrak{c}$.
        \item \textbf{Second case:} $b<0$ and $c>0$.  
    Let $B(\lambda)=-\lambda^{\,n+m}+|b|\lambda^m-|c|=0$. Then $|b|\lambda^m=\lambda^{\,n+m}+|c|$, and hence $-b\lambda^m=\lambda^{\,n+m}+c$.  
    By Theorem~2.8 in \cite{barrera2022number}, one has $\lambda\in\{\mathfrak{b_1},\mathfrak{b_2}\}$ or $\lambda=\mathfrak{b}$.

    \begin{itemize}
        \item If $\lambda=\mathfrak{b_1}$, then $w_1=0$ and $w_2=\pi$, whence $w_*(\mathfrak{b_1})=-m/2$.  
        The interval $(P_*+w_*(0,\mathfrak{b_1}),\, P_*-w_*(0,\mathfrak{b_1}))$ has length $m$.  
        Since $P_*-w_*(\mathfrak{b_1})$ is an integer equal to $0$, it follows that $P_*+w_*(\mathfrak{b_1})$ is also an integer.  
        Thus, the number of integers contained in this interval is $m-1$ if $P_*$ is not an integer, and $m$ if $P_*$ is an integer.  
        Therefore $j=m-1$, i.e. $\lambda=z_{m-1}=\mathfrak{b_1}$, or $j=m$, i.e. $\lambda=z_m=\mathfrak{b_1}$.

        Furthermore, $z_m=\mathfrak{b_2}$, since $z_m<\mathfrak{b_2}$ would imply $ \mathfrak{b_1}=z_{m-1}<z_m<\mathfrak{b_2},$ which is impossible because this interval lies outside the triangle region.  
        The same argument shows that $z_{m+1}=\mathfrak{b_2}$: if $z_{m+1}<\mathfrak{b_2}$, then $\mathfrak{b_1}=z_m<z_{m+1}<\mathfrak{b_2}$, again contradicting that this interval lies outside the triangle region.

        \item If $\lambda=\mathfrak{b}$, then by the Bohl theorem for harmonic trinomials in \cite{barrera2022number}, it has to $j=m$, that is, $\lambda=z_m=\mathfrak{b}$.  
        Since at $\lambda=\mathfrak{b}$ the function $B(\lambda)=-\lambda^{\,n+m}+|b|\lambda^m-|c|$ has a tangency point with the $x$-axis, we conclude that $\lambda=\mathfrak{b}$ is a double root of $B$.  
        Because $h(\lambda)=-B(\lambda)$, it follows that $\lambda=\mathfrak{b}$ is also a double root of $h$.
    \end{itemize}
        \item \textbf{Third case:} $b<0$ and $c<0$.  
    Let $A(\lambda)=\lambda^{\,n+m}-|b|\lambda^m-|c|=0$. Then $\lambda^{\,n+m}=|b|\lambda^{m}+|c|$, and therefore $\lambda^{\,n+m}=-b\lambda^m-c$.  
    By Theorem~2.8 in \cite{barrera2022number}, we obtain $\lambda=\mathfrak{a}$. If $\lambda=\mathfrak{a}$, then $w_1=\pi$ and $w_2=0$, whence $w_*(\mathfrak{a})=(n+m)/2$.  
    The interval $(P_*+w_*(0,\mathfrak{a}),\, P_*-w_*(0,\mathfrak{a}))$ has length $n+m$.  
    Since $P_*-w_*(\mathfrak{a})$ is an integer, $P_*+w_*(\mathfrak{a})$ must also be an integer.  
    Thus, the number of integers contained in the interval is $n+m-1$, and therefore $j=n+m-1$, i.e. $\lambda=z_{n+m-1}=\mathfrak{a}$.
    \end{enumerate}
\end{proof}

\begin{example}
    Let $h=z^5-6\,\overline{z}^{\,2}+1$, the real roots are 
    $z_1\approx -0.405991$, $z_2=\mathfrak{b_1}\approx 0.410624$, and $z_3=\mathfrak{b_2}\approx 1.784863$. For $h=z^3-\frac{3}{2}\,\overline{z}^{\,2}+\frac{1}{2}$, the real roots are $z_2=z_3=\mathfrak{b}=1$ (a double root).
\end{example}

We now present a result that is not mentioned in \cite{theobald2016norms} but is necessary for subsequent results in the context of harmonic trinomials.

\begin{lemma}
\label{reescalamiento}
    Let $h(z)=z^{n+m}+b\,\overline{z}\,^m+c$ with $b\in\mathbb{C}$ and $c\in\mathbb{C}^*$.  
    Then the moduli of the roots of $h$ can be rescaled in such a way that their ordering is preserved.
\end{lemma}

\begin{proof}
    Let $h(z)=z^{\,n+m}+b\,\overline{z}\,^{\,m}+c=0$ with $b\in\mathbb{C}$ and $c\in\mathbb{C}^*$.  
    We write $c=|c|e^{i\theta}$ and define $\zeta=\frac{z}{|c|^{1/(n+m)}}$. Then $h(\zeta)=0$ if and only if $\zeta^{\,n+m}+\frac{b}{|c|^{\,n/(n+m)}}\,\overline{\zeta}\,^{\,m}+e^{i\theta}=0$,
    which is again a harmonic trinomial.  

    Moreover, if $z_1, z_2, \ldots, z_k \in \mathbb{C}^*$ denote the roots of $h$ ordered so that 
    $|z_1|\leq|z_2|\leq \cdots \leq |z_k|$, where $z_k$ represents the last root in this ordering,  
    the above change of variables induces a rescaling of the moduli of the roots while preserving their ordering.
\end{proof}

\begin{theorem}
\label{teo_complementos}
For a fixed $c\in\mathbb{C}^*$ and consider the parametric family of harmonic trinomials
$h_b(z)=z^{n+m}+b\,\overline{z}\,^m+c$, with $b\in\mathbb{C}$. For $j\in\{1,\dots,k-1\}\setminus\{m\}$, the following holds:
\begin{equation*}
    \label{conditions_Uj_par_impar}
    \begin{aligned}
        &\text{If $n+j$ is even, then } h_b\in U_j \text{ if and only if } b\notin \mathcal{R}^{\text{even}}.\\[1mm]
        &\text{If $n+j$ is odd, then } h_b\in U_j \text{ if and only if } b\notin \mathcal{R}^{\text{odd}}.
    \end{aligned}
\end{equation*}
    
    In particular, the set $\{b\in\mathbb{C}^*: h_b\in U_j\}$ is disconnected.  
    This property remains valid for the set $\{b\in\mathbb{C}: h_b\in U_j\}$ when the sets $U_j$ are considered in $\mathbb{C}\times\mathbb{C}^*$.  

    For $U_m$, the conditions for $n+j$ even or odd also hold, with the additional modification that $h_b\in U_m$ whenever there exists $\lambda\in\mathbb{R}^+$ such that every point $h_b(\lambda)\in\mathbb{R}$ in the interval $(|z_{m}|,|z_{m+1}|)$ lies outside the triangle region.
\end{theorem}

\begin{remark}
    Theorem \ref{teo_complementos} is similar to Theorem 4.9 in \cite{theobald2016norms}; however, we present the proof using a different and more detailed strategy for harmonic trinomials.
\end{remark}

\begin{proof}
Let $z_1,z_2,\ldots,z_k\in\mathbb{C}^*$ be all the possible roots of $h_b$ ordered by their moduli, that is, $|z_1|\le |z_2|\le \cdots \le |z_k|,$
where $z_k$ is the last root in this ordering and where $k\le n+3m$.  
By Theorem~\ref{teo:union_rayos}, it is sufficient to consider the case $b\in \mathcal{R}(n,m,c)$.

By Lemma~\ref{reescalamiento}, the moduli of the roots of $h_b$ may be rescaled while preserving their order, so we may assume that $|c|=1$.  
Furthermore, when $|c|=1$, the polynomial $h_b(z)=z^{n+m}+b\,\overline{z}^{\,m}+c$ is Egerváry equivalent to another harmonic trinomial obtained by a rotation of the arguments; hence we may also assume $c=1$.

For $b=0$ all roots satisfy $z=(-1)^{1/(n+m)}$ and they have modulus $1$. Therefore, $h_{b}\notin U_j \text{ for every } j\in\{1,\ldots,k-1\}.$ Hence we may assume from now on that $b\in\mathbb{C}^*$.

Recall that the interval $(P_*+w_*(0,v)) \,\cup\, (P_*-w_*(0,v))$ is related to the number of roots of modulus $v$.  
Whenever $v^{n+m}$, $|b|v^m$, and $|c|$ form the sides of a nondegenerate triangle, we have, for each $j\in\{1,\ldots,k-1\}\setminus\{m\}$,
$$
h_b\in U_j \text{ with $j$ even } \iff P_*\notin\mathbb{Z},\text{\qquad and \qquad}h_b\in U_j \text{ with $j$ even } \iff P_*\notin\mathbb{Z}.
$$

In our case $\alpha=0$ and $\gamma=0$, so $P_*=\frac{(n+m)\beta-(n+2m)\pi}{2\pi}$. Solving for $\beta$ gives
\[
\beta=\frac{(2P_*+n+2m)\pi}{n+m}
      =\frac{l\pi}{n+m},
\]
where $l := 2P_*+n+2m\in\mathbb{Z}$.  
This corresponds precisely to the definition of a union of rays $b\in \mathcal{R}(n,m,c)$ given in \eqref{rayos}.  
Therefore, $b\in \mathcal{R}(n,m,c)$ if only if $\beta=\frac{l\pi}{n+m},\, l\in\{0,\ldots,2(n+m)-1\}$.

With this, $P_*$ becomes
\begin{equation}
\label{P_*}
    P_*=\frac{l-(n+2m)}{2}.
\end{equation}

Moreover, since $b\in \mathcal{R}(n,m,c)$ satisfies
\[
    b\in\mathcal{R}^{\mathrm{even}}
    \iff \frac{(n+m)\beta}{\pi} \ \text{is even} 
    \iff l \text{ is even,\quad and}
\]
\[
    b\in\mathcal{R}^{\mathrm{odd}}
    \iff \frac{(n+m)\beta}{\pi} \ \text{is odd}
    \iff l \text{ is odd}.
\]

From \eqref{P_*} the following cases occur:  
\[
l \text{ and } n \text{ have the same parity } 
\quad\iff\quad P_*\in\mathbb{Z},\quad\text{ and }
\]
\[
l \text{ and } n \text{ have different parity } 
\quad\iff\quad P_*\notin\mathbb{Z}.
\]
We also have the following:
\begin{itemize}
    \item If $l$ is not even, then $b\notin \mathcal{R}^{\mathrm{even}}$.
        \begin{itemize}
            \item $l$ is odd and $n$ is even, which implies $P_*\notin\mathbb{Z}$, hence $h_b\in U_j$ with $j$ even.
            \item Or $l$ is odd and $n$ is odd, which implies $P_*\in\mathbb{Z}$, hence $h_b\in U_j$ with $j$ odd.
        \end{itemize}

    \item If $l$ is not odd, then $b\notin \mathcal{R}^{\mathrm{odd}}$.
        \begin{itemize}
            \item $l$ is even and $n$ is even, which implies $P_*\in\mathbb{Z}$, hence $h_b\in U_j$ with $j$ odd.
            \item Or $l$ is even and $n$ is odd, which implies $P_*\notin\mathbb{Z}$, hence $h_b\in U_j$ with $j$ even.
        \end{itemize}
\end{itemize}

Therefore,
$$
n+j \text{ \,\,even}: h_b\in U_j \iff b\notin\mathcal{R}^{even},
$$
$$
n+j \text{ \,\,odd}: h_b\in U_j \iff b\notin\mathcal{R}^{odd}.
$$

\medskip
The non-connectedness of $U_j$ along the $\mathbb{C}$–slice $(H)_c$ for $c=1$ follows directly from the fact that 
\[
\mathbb{C}\setminus \mathcal{R}^{\mathrm{odd}}(n,m,c)
\quad\text{and}\quad
\mathbb{C}\setminus \mathcal{R}^{\mathrm{even}}(n,m,c)
\]
are both non-connected sets.

\medskip
It remains to consider the case $j=m$.  
The previous argument still applies, except that by Theorems~\ref{BohlArmonico} and~\ref{RaicesReales}, we additionally have $h_b\in U_m$ whenever there exists $\lambda\in\mathbb{R}^+$ such that $h_b(\lambda)$ lies outside the triangle region.
\end{proof}

\begin{example}
Let $h=z^{5}+6\,\overline{z}^{\,3}+1$, with $j=\{1,2,\dots,10\}\backslash\{3\}.$ We have:\\
For $j = 1,5,7,9$, $(n+j)$ is odd, the condition holds, that is, $h \in U_j$ if and only if $b \notin \mathcal{R}^{\mathrm{odd}}$.\\
For $j = 2,4,6,8,10$, $(n+j)$ is even, and the condition does not hold.\\
For $j=m$, $(n+m)=5$ is odd. Hence $h\in U_3 \quad\text{if and only if}\quad b\notin \mathcal{R}^{\mathrm{odd}}$ (see Figure~\ref{fig:EjemploRayosPares}).  
However,
\begin{equation*}
    h \in U_0 \cap U_1 \cap U_3 \cap U_5 \cap U_7 \cap U_9 \cap U_{11}.
\end{equation*}
Moreover, there exists $\lambda\in\mathbb{R}^+$ such that every point $h_b(\lambda)\in\mathbb{R}$ in the interval $(|z_3|, |z_4|)$ lies outside the triangle region.\\
To verify that the conditions of the Theorem \ref{teo_complementos} are satisfied, the moduli of all the roots of $h=z^{5}+6\,\overline{z}^{\,3}+1$ are listed.
\begin{equation*}
\begin{array}{llllll}
    0.54163, & 0.55589, &0.55589, &2.43641, & 2.43641, & 2.44415,\\ 2.44415, & 2.45478, &
    2.45478, & 2.46209, & 2.46209.
\end{array}
\end{equation*}

\begin{figure}[h!]
\centering

\begin{minipage}{0.35\textwidth}
    \includegraphics[width=\linewidth]{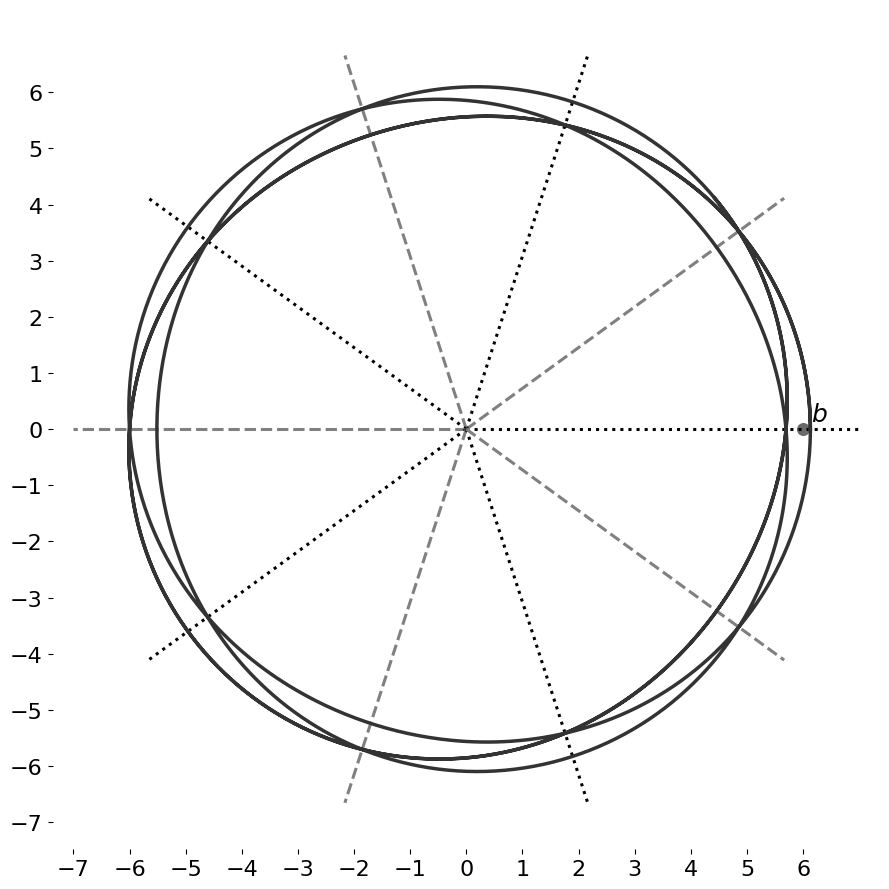}
\end{minipage}
\hfill
\begin{minipage}{0.61\textwidth}
    \caption{\small{Trochoid associated to $h = z^{5} + 6\,\overline{z}^{3} + 1$, taking a complex root with modulus $v \approx 0.55589$. The dotted lines represent the even rays ($\mathcal{R}^{\mathrm{even}}$) and the dashed lines represent the odd rays ($\mathcal{R}^{\mathrm{odd}}$). We observe that $b = 6 \notin \mathcal{R}^{\mathrm{odd}}$.}}
    \label{fig:EjemploRayosPares}
\end{minipage}

\end{figure}

\end{example}
The following theorem is a result of Greenfield and Drucker, who provide an explicit formula for the discriminant of monic univariate trinomials. In our case, we are interested in finding double roots of harmonic trinomials. However, a notion of discriminant has not yet been developed for harmonic trinomials. Therefore, we can only offer an interpretation of Theorem 4 in \cite{GREENFIELD1984105}.
\begin{theorem} (Greenfield and Drucker, 1984)
    Let $h(z)=z^{\,n+m}+b z^{\,m}+c$ be a trinomial with $b,c\in\mathbb{C}$, where $n$ and $m$ are relatively prime integers.  
    Then the discriminant $D(h)$ of $h$ is given by
    \[
        D(h)=(-1)^{\frac{(n+m)(n+m-1)}{2}}\, c^{\,m-1}
        \left(c^{\,n}(n+m)^{\,n+m}-(-1)^{\,n+m} b^{\,n+m} n^{\,n} m^{\,m}\right).
    \]
\end{theorem}
Below we present the results on multiple roots of a harmonic trinomial that have no analogue in the non-harmonic trinomial setting in \cite{theobald2016norms}.
\begin{lemma}
    Let $H(z)=P(z)+\overline{Q(z)}$ be a harmonic polynomial, where $P$ and $Q$ are complex polynomials.  
    Let $\zeta$ be a root of $H$. Then:
    \begin{enumerate}
        \item if $J_H(\zeta)\neq 0$, then $\zeta$ is a simple root;
        \item if $J_H(\zeta)=0$, then $\zeta$ has multiplicity at least $2$.
    \end{enumerate}
\end{lemma}
The proof follows from straightfoward computations and we omit it.

\begin{corollary}
\label{MultRoot}
    Let $h(z) = az^{n+m} + b\overline{z}\,^m + c$ be a harmonic trinomial, and let $\zeta$ be a root of $h$. If $J_h(\zeta) = 0$, then $\zeta$ is a root of multiplicity at least 2.
\end{corollary}

\begin{corollary}
\label{frontera}
    If a harmonic trinomial $h = z^{n+m} + b \overline{z}\,^m + c$ with $b,c \in \mathbb{C}^*$ has a double root $\zeta_j = \zeta_{j+1}$, then $\theta = \frac{\gamma}{n+m} \mod 2\pi,$ where $\theta = \arg(\zeta)$ and $\gamma = \arg(c)$.
\end{corollary}
\begin{proof}
    Let $\zeta_j$ be a root of $h$, so that $b(\overline{\zeta_j})^m = -(\zeta_j)^{n+m} - c$. Then, $|b||\overline{\zeta_j}|^m = |\zeta_j^{n+m} + c| \leq |\zeta_j|^{n+m} + |c|$. Since $\zeta_j$ is a double root of $h$, the equality $|\zeta_j^{n+m} + c| = |\zeta_j|^{n+m} + |c|$ occurs. Hence $\zeta_j^{n+m}\cdot\overline{c} + \overline{\zeta_j}\,^{n+m}\cdot c = 2|\zeta_j^{n+m}||c|$.
   Now, let $\zeta_j = |\zeta_j| e^{i\theta}$, then
   \begin{equation}
   \label{PropComp}
       \begin{array}{ccc}
             e^{i(n+m)\theta} \cdot \overline{c} + e^{-i(n+m)\theta} \cdot c & = & 2|c| 
       \end{array}
   \end{equation}
   On the other hand, we have
   \begin{equation}
   \label{PropComp2}
       e^{i(n+m)\theta} \cdot \overline{c} + e^{-i(n+m)\theta} \cdot c = 2\Re(e^{-i(n+m)\theta} \cdot c).
   \end{equation}
   From \eqref{PropComp} and \eqref{PropComp2}, we get $\Re(e^{-i(n+m)\theta} \cdot c) = |c|$. Let $e^{-i(n+m)\theta} = \cos((n+m)\theta) - i \sin((n+m)\theta)$ and $c = c_1 + i c_2$, so that $\Re([\cos((n+m)\theta) - i \sin((n+m)\theta)](c_1 + i c_2)) = \sqrt{c_1^2 + c_2^2}$ and $\theta = \frac{\arctan\big(\frac{c_2}{c_1}\big)}{n+m}$.
\end{proof}
Finally, we provide a reinterpretation of Corollary~4.13 in \cite{theobald2016norms} for harmonic trinomials.

\begin{corollary}
\label{radio}
    For a fixed $c\in\mathbb{C}^*$, let $h_b = z^{n+m} + b\overline{z}\,^m + c$ be a parametric family of harmonic trinomials with parameter $b \in \mathbb{C}$. Then:
    \begin{enumerate}
        \item for $n+j$ even: $h_b \in U_m$ if and only if $b \notin \mathcal{R}^{\mathrm{even}} \cap B_\rho(0)$,
        \item for $n+j$ odd: $h_b \in U_m$ if and only if $b \notin \mathcal{R}^{\mathrm{odd}} \cap B_\rho(0)$,
    \end{enumerate}
    where $\rho = |c|^{n/(n+m)} \Big( \big(\frac{m}{n+2m}\big)^{n/(n+m)} + \big(\frac{n}{n+2m}\big) \big(\frac{n+2m}{m}\big)^{m/(n+m)} \Big)$ and $B_\rho(0)$ is the closed disk of radius $\rho$ centered at the origin.
\end{corollary}

\begin{remark}
    The following proof uses ideas similar to those in Corollary 4.13 of \cite{theobald2016norms}, but includes specific details when dealing with the triangle region associated with a harmonic trinomial.
\end{remark}

\begin{proof} 
In Theorem \ref{teo_complementos}, for $U_m$ (recall that $j \in \{1,\dots,k-1\} \setminus \{m\}$), the conditions in \eqref{conditions_Uj_par_impar} are still satisfied with the modification that additionally $h_b \in U_m$ if there exists $v \in \mathbb{R}^+$ such that $h_b$ does not belong to the region forming triangles. Hence, $(U_m)^c$ restricted to the $\mathbb{C}$-slice given by fixed $c \in \mathbb{C}^*$ is the subset of $\mathcal{R}^{\text{odd}}$ (resp. $\mathcal{R}^{\text{even}}$) where $h_b$ belongs to the region forming triangles.

Moreover, we know that the line segments $v^{n+m}, |b|v^m,$ and $|c|$ associated with $h_b$ are independent of the arguments of the coefficients. This property remains valid in an open subset of $U_m$, which lies outside the triangle-forming region; in particular, it satisfies $|b|v^m > v^{n+m} + |c|$. This inequality persists under an increase in $|b|$. Indeed, for $|b_0| > |b|$, we have $|b_0|v^m > |b|v^m > v^{n+m} + |c|$. That is, under an increase in $|b|$, the $h_b$ such that $h_b \in U_m$, where $U_m$ is open, remain outside the region forming triangles.

Therefore, $(U_m)^c$ along the $\mathbb{C}$-slice $(H)_c$ for fixed $c$ is given by $\mathcal{R}^{\text{odd}}$ (resp. $\mathcal{R}^{\text{even}}$) intersected with a closed disk. 

By the reasoning in Corollary \ref{frontera}, the boundary of this disk is determined by choosing $|b| > 0$ such that $h = z^{n+m} + |b|\overline{z}\,^m + |c|$ has a double root $\zeta_m = \zeta_{m+1}$. Since $h_{|b|_{\epsilon>0}}$ lies outside the triangle region at $\zeta_m$ for each $\epsilon>0$, from $h^{\zeta_j}(\phi) = b + \zeta_j^n e^{i(n+2m)\phi} + |c| \zeta_j^{-m} e^{i(\gamma + m\phi)} = 0$ for some $\phi \in [0,2\pi)$ in the proof of Theorem \ref{RaizDoble}, differentiating gives $\frac{\partial}{\partial \phi} h^{\zeta_j}(\phi) = \zeta_j^n i(n+2m)e^{i(n+2m)\phi} + |c| \zeta_j^{-m} (im) e^{i(\gamma + m\phi)} = 0$, so
$\zeta_j^{n+m} = \frac{-|c| m e^{i(\gamma + m\phi)}}{(n+2m) e^{i(n+2m)\phi}}$,
which gives
\begin{equation}
\label{root_aj}
|\zeta_j| = \sqrt[n+m]{\frac{|c| m}{n+2m}}.
\end{equation}
Also, from the proof of Theorem \ref{RaizDoble}, we have $b = \frac{n+m}{m} \zeta_j^n e^{i(n+2m)\phi}$, hence $|\zeta_j|^n = \frac{|b| m}{n+m}$.
Thus, $|\zeta_j|^n = \frac{|b| m}{n+m}$ and $|\zeta_j|^{n+m} = \frac{|c| m}{n+2m}$. Dividing these gives
\begin{equation*}
|\zeta_j|^m = \frac{|c| (n+m)}{|b| (n+2m)},
\end{equation*}
and therefore
\begin{equation}
\label{bajm}
|b \zeta_j^m| = \frac{|c| (n+m)}{n+2m} = \frac{|c| n}{n+2m} + \frac{|c| m}{n+2m}.
\end{equation}
On the other hand, from \eqref{root_aj} we can also write
\begin{equation}
\label{bajm2}
|b \zeta_j^m| = |b| |\zeta_j|^m = |b| \Big(\frac{|c| m}{n+2m}\Big)^{m/(n+m)}.
\end{equation}
Comparing \eqref{bajm} and \eqref{bajm2}, we obtain
\[
|b| \Big(\frac{|c| m}{n+2m}\Big)^{m/(n+m)} = \frac{|c| n}{n+2m} + \frac{|c| m}{n+2m},\quad\text{ so }
\]
\begin{equation*}
|b| = \frac{\frac{|c| n}{n+2m} + \frac{|c| m}{n+2m}}{\Big(\frac{|c| m}{n+2m}\Big)^{m/(n+m)}} 
= |c|^{n/(n+m)} \Bigg( \Big(\frac{m}{n+2m}\Big)^{n/(n+m)} + \Big(\frac{n}{n+2m}\Big) \Big(\frac{n+2m}{m}\Big)^{m/(n+m)} \Bigg).
\end{equation*}
\end{proof}

\begin{example}
    Let $h_1(z)=z^8+b\overline{z}\,^3+\frac{1}{2}$, $h_2(z)=z^7+b\overline{z}\,^2+2$, and $h_3(z)=z^5+b\overline{z}+1$ as in Example \ref{example_hip}. Then $h_1, h_2,$ and $h_3$ each have two roots with the same modulus if and only if $b$ lies in the intersection of $\mathcal{R}^{\mathrm{odd}}$ (respectively $\mathcal{R}^{\mathrm{even}}$) with a closed (see Figures \ref{fig:CurvaRayos1}, \ref{fig:CurvaRayos2}, \ref{fig:CurvaRayos3}).

    \begin{figure}[h!] 
    \centering 
    \begin{subfigure}[b]{0.32\textwidth} 
        \centering
        \includegraphics[width=\textwidth]{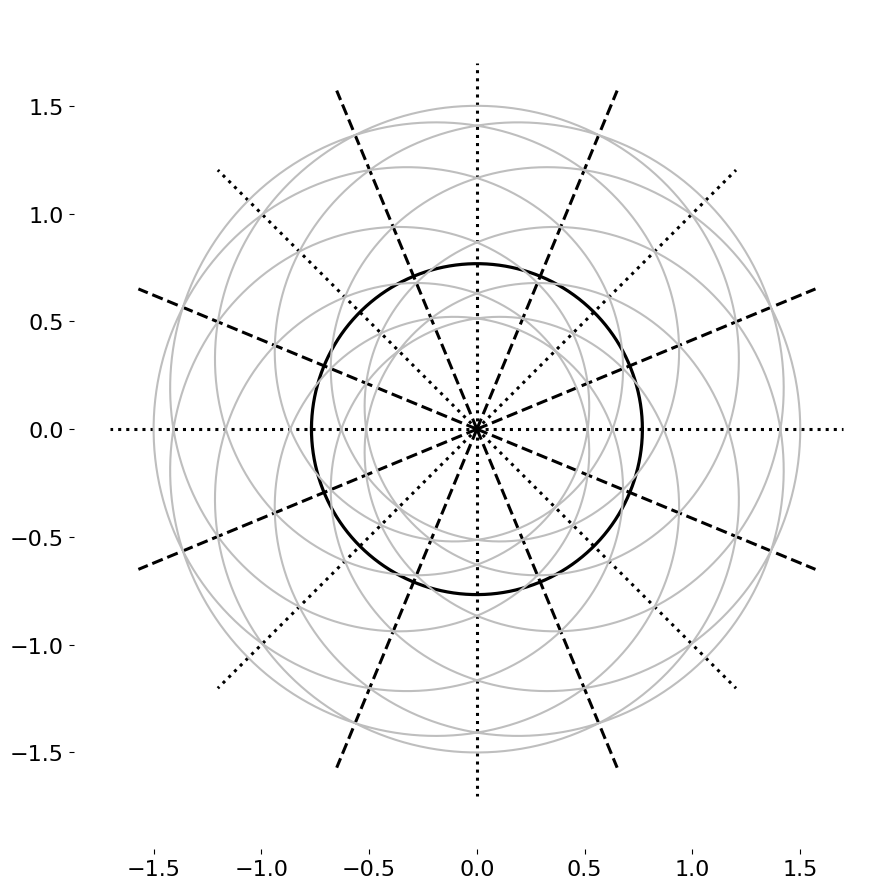} 
        \caption{\small{Possible intersection points of the rays corresponding to $h_1 = z^8 + b\overline{z}\,^3 + \frac{1}{2}$ and a closed disk of radius $\rho_1\approx 0.7676$.}}
        \label{fig:CurvaRayos1}
    \end{subfigure}
    \hfill 
    \begin{subfigure}[b]{0.32\textwidth}
        \centering
        \includegraphics[width=\textwidth]{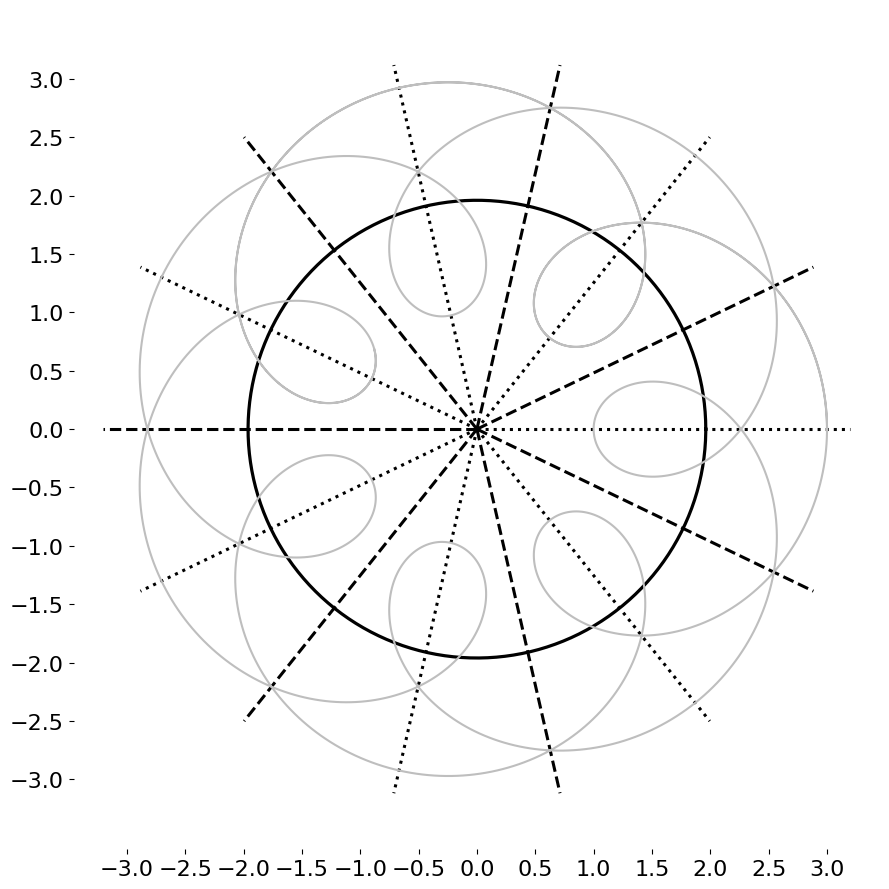}
        \caption{\small{Possible intersection points of the rays corresponding to $h_1 = z^7 + b\overline{z}\,^2 + 2$ and a closed disk of radius $\rho_2\approx1.9616$.}}
        \label{fig:CurvaRayos2}
    \end{subfigure}
    \hfill
    \begin{subfigure}[b]{0.32\textwidth}
        \centering
    \includegraphics[width=\textwidth]{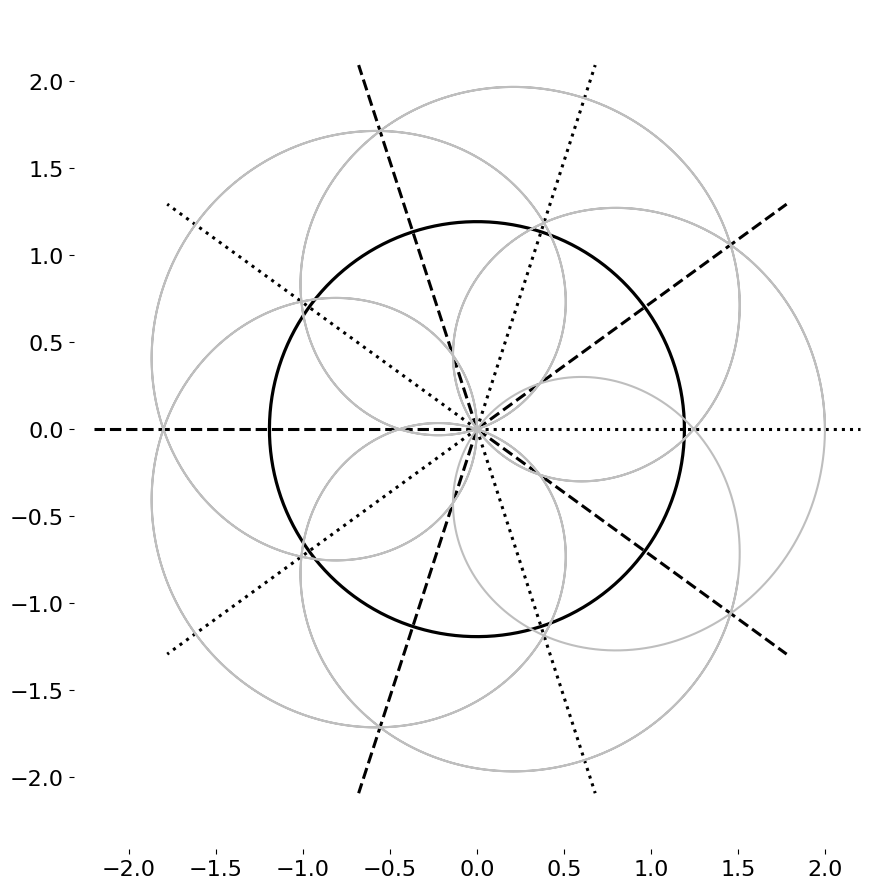}
        \caption{\small{Possible intersection points of the rays corresponding to $h_1 = z^5 + b\overline{z} + 1$ and a closed disk of radius $\rho_3\approx1.2052$.}}
        \label{fig:CurvaRayos3}
    \end{subfigure}
    \caption{\small{On the gray curves are traced the trajectories of the fiber functions 
$h_1^1 - b$, $h_2^1 - b$, and $h_3^1 - b$, each having a root of modulus 
$v = 1$, together with circles of radii $\rho_1$, $\rho_2$, and $\rho_3$, 
respectively. The dotted lines correspond to the even rays 
($\mathcal{R}^{\mathrm{even}}$), while the dashed lines correspond to the odd 
rays ($\mathcal{R}^{\mathrm{odd}}$).
}} 
    \label{fig:multiple}
\end{figure}
\end{example}

\begin{remark}
 The radius of the closed disks associated with the trinomials in \cite{theobald2016norms} is larger than the radius presented in Corollary \ref{radio} for harmonic trinomials. Since the Jacobian may vanish only within these disks, we interpret this as indicating that the region in which singular zeros can occur is smaller for harmonic trinomials than for non-harmonic ones. 
\end{remark}

\section*{Acknowledgements}
The authors acknowledge the support received by Universidad Autónoma de Yucatán.

\section*{Funding}

W. Barrera was partially supported by SECIHTI under grant SNI 45382. J.P. Navarrete was partially supported by SECIHTI under grant SNI 44867. L. Campa acknowledge the support of SECIHTI under Becas Nacionales program (540598).

\section*{Declarations}

\subsection{Conflict of interest} The author declare they have no conflict of interests.

\newpage

\bibliographystyle{abbrv}
\bibliography{referenciasarticulo}

\end{document}